\documentclass[11pt]{article}
\pdfoutput=1
\usepackage{smile}
\usepackage[dvipsnames]{xcolor}

\usepackage{caption}
\usepackage{subcaption}
\graphicspath{ {./images/} }

\begin{document}
\title{A General Continuous-Time Formulation of Stochastic ADMM and Its Variants}
\author{
Chris Junchi Li
\thanks{University of California, Berkeley, CA 94720-1776, USA; email: \url{junchili@berkeley.edu}}
}
\date{\today}
\maketitle

\begin{abstract}
Stochastic versions of the alternating direction method of multiplier (ADMM) and its variants play a key role in many modern large-scale machine learning problems. In this work, we introduce a unified algorithmic framework called generalized stochastic ADMM and investigate their continuous-time analysis. The generalized framework widely includes many stochastic ADMM variants such as standard, linearized and gradient-based ADMM. Our continuous-time analysis provides us with new insights into stochastic ADMM and variants, and we rigorously prove that under some proper scaling, the trajectory of stochastic ADMM weakly converges to the solution of a stochastic differential equation with small noise. Our analysis also provides a theoretical explanation of why the relaxation parameter should be chosen between 0 and 2.
\end{abstract}

\section{Introduction}\label{sec1}
For modern industrial-scale machine learning problems, with the massive amount of data that are not only extremely large but often stored or even collected in a distributed manner, stochastic first-order methods almost become one of the default choices due to their excellent performance for the online streaming large-scale datasets.
The stochastic version of the alternating direction method of multiplier (ADMM) algorithms is a popular approach to handling this distributed setting, especially for regularized empirical risk minimization.
Consider the following stochastic optimization problem:
\begin{equation}\label{eq1.1}
\underset{x \in \mathbb{R}^{d}}{\operatorname{minimize}}~
V(x)
:=
f(x)+g(A x)
=
\mathbb{E}_{\xi} f(x, \xi)+g(A x)
\end{equation}
where $f(x)=\mathbb{E}_{\xi} f(x, \xi)$ with $f(x, \xi)$ as the loss incurred on a sample $\xi$, $f: \mathbb{R}^{d} \rightarrow \mathbb{R} \cup$ $\{+\infty\}$, $g: \mathbb{R}^{m} \rightarrow \mathbb{R} \cup\{+\infty\}$, $A \in \mathbb{R}^{m \times d}$, and both $f$ and $g$ are convex.
The alternating direction method of multiplier (ADMM)~\cite{admm3boyd} rewrites~\eqref{eq1.1} as a constrained optimization problem
\begin{equation}\label{eq1.2}
\begin{aligned}
&
\underset{x \in \mathbb{R}^{d}, z \in \mathbb{R}^{m}}{\operatorname{minimize}}
&&
\mathbb{E}_{\xi} f(x, \xi) + g(z)
\\&
\text{subject to}
&&
A x-z = 0
\end{aligned}\end{equation}
Here and throughout the rest of the paper, we overload $f$ to ease the notation, i.e.~we adopt the two-argument $f(\cdot, \xi)$ for the stochastic instance and one-argument $f(\cdot)$ for its expectation.
Note the classical setting of linear constraint $A x+B z=c$ can be reformulated as $z=A x$ by a linear transformation when $B$ is invertible.
In the batch learning setting, $f(x)$ can be approximated by the empirical risk $f_{emp}=\frac{1}{N} \sum_{i=1}^{N} f(x, \xi_{i})$.
However, the computation cost for minimizing $f_{emp}$ with a large amount of samples is significantly high and the efficiency is limited under the time and resource constraints.
In the stochastic setting, at each step $x_{k}$ is updated based on a small batch of samples (or even one) $\xi_{k}$ instead of a large batch or full batch.

Introducing stochasticity to ADMM~\cite{admm3boyd} is in parallel to incorporating noise into gradient descent~\cite{WangBanerjeeIMCL2012,suzuki2013dual,pmlr-v28-ouyang13}.
At iteration $k+1$, a sample $\xi_{k+1}$ is randomly drawn as an independent realization of $\xi$.
Analogous to stochastic gradient descent, our \emph{stochastic ADMM} (\emph{sADMM}) introduced in~\cite{suzuki2013dual,pmlr-v28-ouyang13} performs the following updates:
\begin{subequations}
\begin{align}
x_{k+1}
&=
\underset{x}{\operatorname{argmin}}\left\{f(x, \xi_{k+1})+\frac{\rho}{2}\left\|A x-z_{k}+u_{k}\right\|_{2}^{2}\right\}
\label{eq1.3a}
\\
z_{k+1}
&=
\underset{z}{\operatorname{argmin}}\left\{g(z)+\frac{\rho}{2}\left\|\alpha A x_{k+1}+(1-\alpha) z_{k}-z+u_{k}\right\|_{2}^{2}\right\}
\label{eq1.3b}
\\
u_{k+1}
&=
u_{k}+\left(\alpha A x_{k+1}+(1-\alpha) z_{k}-z_{k+1}\right)
\label{eq1.3c}
\end{align}\end{subequations}
$\rho>0$ is called the penalty parameter~\cite{admm3boyd}.
Here $\alpha \in(0,2)$ is introduced as a relaxation parameter~\cite{admm3boyd}.
The algorithm is split into different regimes of $\alpha$ as follows:

\begin{itemize}
\item
$\alpha=1$, corresponding to the \emph{standard stochastic ADMM};
\item
$\alpha>1$, corresponding to the \emph{over-relaxed stochastic ADMM};
\item
$\alpha<1$, corresponding to the \emph{under-relaxed stochastic ADMM}.
\end{itemize}

Historically, the above relaxation schemes were introduced and analyzed in~\cite{eckstein1992douglas}, and the experiments in~\cite{eckstein1998operator} suggest that the over-relaxed regime of $\alpha>1$ can improve the convergence rate.
The acceleration phenomenon in over-relaxed ADMM has been discussed from the continuous perspective in deterministic setting~\cite{pmlr-v97-yuan19c,francca2023nonsmooth}, where relaxed the ADMM algorithm accelerates by a factor of $\alpha\in (1,2)$ in its convergence rate.

Since the emergence of ADMM, many variants have been introduced recently for solving a variety of optimization tasks.
We focus on two ADMM variants (in our stochastic setting) to cater to the need in application:

\begin{enumerate}
\item[(i)]
In the linearized ADMM~\cite{goldfarb2013fast,lin2011linearized,ouyang2015accelerated},~\eqref{eq1.4a} replaces~\eqref{eq1.3a} where
\begin{equation}
x_{k+1}
:=
\underset{x}{\operatorname{argmin}}\left\{
f(x, \xi_{k+1})
+
\frac{\tau}{2}\left\|
x-(x_{k}-\frac{\rho}{\tau} A^{\top}(A x_{k}-z_{k}+u_{k}))
\right\|_{2}^{2}
\right\}
\label{eq1.4a}
\end{equation}
In words, the augmented Lagrangian function is approximated by linearizing the quadratic term of $x$ in~\eqref{eq1.3a} plus the addition of a proximal term $\frac{\tau}{2}\|x-x_{k}\|_{2}^{2}$;

\item[(ii)]
In the gradient-based ADMM~\cite{Condat-2013,Vu-2013,Davis-Yin-3OS,Ma-Zhang-EGADM-2013},~\eqref{eq1.5a} replaces~\eqref{eq1.3a} where
\begin{equation}\label{eq1.5a}
x_{k+1}
:=
x_{k}-\frac{1}{\tau}\left(f^{\prime}\left(x_{k}, \xi_{k+1}\right)+\rho A^{\top}\left(A x_{k}-z_{k}+u_{k}\right)\right)
\end{equation}
In words, instead of solving $x$-subproblem~\eqref{eq1.3a} accurately, we apply one step of gradient descent with stepsize $1 / \tau$.
\end{enumerate}

In this paper we formulate a general scheme, called \emph{Generalized stochastic ADMM} (\emph{G-sADMM)}, to accommodate and unify all these variants of stochastic ADMM:
\begin{subequations}
\begin{align}
x_{k+1}
&=
\underset{x}{\operatorname{argmin}}~\hat{\mathcal{L}}_{k+1}\left(x, z_{k}, u_{k}\right)
\label{eq1.6a}
\\
z_{k+1}
&=
\underset{z}{\operatorname{argmin}}\left\{g(z)+\frac{\rho}{2}\left\|\alpha A x_{k+1}+(1-\alpha) z_{k}-z+u_{k}\right\|_{2}^{2}\right\}
\label{eq1.6b}
\\
u_{k+1}
&=
u_{k}+\left(\alpha A x_{k+1}+(1-\alpha) z_{k}-z_{k+1}\right)
\label{eq1.6c}
\end{align}\end{subequations}
where the objective $\hat{\mathcal{L}}_{k+1}$ in~\eqref{eq1.6a} for $x$-subproblem is defined as
\begin{equation}\label{eq1.7}
\begin{aligned}
&
\hat{\mathcal{L}}_{k+1}\left(x, z_{k}, u_{k}\right)
\\&:=
(1-\omega_{1}) f(x, \xi_{k+1})+\omega_{1} f^{\prime}\left(x_{k}, \xi_{k+1}\right)\left(x-x_{k}\right)+(1-\omega) \cdot \frac{\rho}{2}\left\|A x-z_{k}+u_{k}\right\|_{2}^{2}
\\
&\quad
+
\omega \rho A^{\top}\left(A x_{k}-z_{k}+u_{k}\right)\left(x-x_{k}\right)+\frac{\tau}{2}\left\|x-x_{k}\right\|_{2}^{2}
\end{aligned}\end{equation}
with parameter $\rho, \tau \geq 0, \alpha \in(0,2)$ and explicitness parameters $\omega_{1}, \omega \in[0,1]$.

\section{Main results}\label{sec2}
We study the continuous-time limit of the stochastic sequence $\{x_{k}\}$ defined in~\eqref{eq1.6a} as $\rho \rightarrow \infty$.
Following the recent line of work~\cite {francca2018admm}~\cite{francca2023nonsmooth} and~\cite {pmlr-v97-yuan19c}, we adopt a (speed-up) time rescaling of $\rho^{-1}$, namely the stepsize for (stochastic) ADMM, and hence we are in the small-stepsize regime as in (stochastic) gradient descent.
Also, the proximal parameter $\tau$ in~\eqref{eq1.7} grows linearly in $\rho$ i.e.~$\tau / \rho \rightarrow c$.
In our setting of continuous limit theory, we focus on a fixed interval by $T$ (referred to as "time"), corresponding to $\lfloor\rho T\rfloor$ discrete steps (referred to as "step") in discrete time.
Our main results can be concluded as follows:

\begin{enumerate}
\item[(i)]
As $\rho \rightarrow \infty$ the stochastic sequence $\{x_{\lfloor\rho t\rfloor}\}$ for $1 \leq t \leq T$ admits a continuous limit $\{X_{t}: t \in[0, T]\}$, a stochastic process that solves the following stochastic differential equation (SDE):
\begin{equation}\label{eq2.1}
\left[c I+\left(\frac{1}{\alpha}-\omega\right) A^{\top} A\right] d X_{t}
=
-
\nabla V(X_{t}) d t
+
\sqrt{\frac{1}{\rho}} \sigma(X_{t}) d W_{t}
\end{equation}
where we recall from~\eqref{eq1.1} that $V(x)=f(x)+g(A x)$.
The right hand of~\eqref{eq2.1} consists of a deterministic part and a stochastic part, where $\sigma(x)$ is some diffusion matrix defined later in~\eqref{eq3.8} and $W_{t}$ is the standard Brownian motion in $\mathbb{R}^{d}$.%
\footnote{Section~\ref{sec3.1} is a summary of the background of the stochastic differential equation and the weak approximation, and the theory on the stochastic modified equation for interested readers.}
The continuous-limit approximation is rigorously characterized by weak convergence%
\footnote{The weak convergence with a uniform bound of $p$-th momentum is equivalent to the convergence in the Wasserstein-$p$ metric. In our case, the result holds for Wasserstein-2 convergence.}
in the sense that applying a test function $\varphi$ of a certain class, the expectation of $\varphi(x_{k})-\varphi\left(X_{k / \rho}\right)$ is uniformly in $k \in[0,\lfloor\rho T\rfloor]$ convergent to zero when $\rho$ tends to infinity as stated in Theorem~\ref{theo3.2}.%
\footnote{We warn the reader that this is towards an orthogonal direction to the convergence of $x_{k}$ to the optimizer $x_{*}$ of the objective function as $k \rightarrow \infty$ for a \emph{fixed} $\rho$.}
The above SDE carries the small parameter $\sqrt{1 / \rho}$ in its diffusion term and is also called stochastic modified equation (SME)~\cite{LTE2019JMLR}, due to the historical reason in numerical analysis~\cite{Milstein1995book,KPSDEbook2011}.

\item[(ii)]
We provide a continuous-time explanation of the effect of relaxation parameter $\alpha \in(0,2)$ in stochastic ADMM, which is the first among stochastic ADMM work to our best knowledge.
In light of our weak convergence of the discrete algorithm $x_{k}$, the residual $r_{k}=A x_{k}-z_{k}$ approximately satisfies $r_{k+1} \approx(1-\alpha) r_{k}$ and converges to 0 at a geometric rate $|1-\alpha|<1$.
The rigorous statement is in Corollary~\ref{coro3.1}.
\end{enumerate}

\subsection{Contributions}\label{sec2.1}
Our contribution to this paper is the first continuous-time analysis of stochastic ADMM.
First, we present a unified stochastic differential equation as the continuous-time model of variants of stochastic ADMM (standard ADMM, linearized ADMM, gradient-based ADMM) in the regime of large $\rho$ under weak convergence.
By showing the effective dynamic is restricted in the $x$-component and a careful analysis of noisy fluctuation, we revealed that these stochastic ADMM variants surprisingly have the essentially same continuous limit and the exact diffusion terms as preconditioned SGD, regardless of the constraint, the distributed nature and the stochasticity as well as various parameters in consideration.

Second, we characterize the time evolution of $\operatorname{std}(x_{k})$ and $\operatorname{std}(z_{k})$ in their continuous-time counterparts and explicitly show that the standard deviation of stochastic ADMM variants has the scaling $\rho^{-1 / 2}$ in stochastic ADMM.
Finally, our theoretical analysis of the continuous model provides a simple justification of the effect of relaxation parameter $\alpha \in(0,2)$ and further sheds light on the principled choices of parameters $c, \alpha, \omega$.

\subsection{Related work}\label{sec2.2}
ADMM is a widely used algorithm for solving problems with separable structures in machine learning, statistics, control, etc.
ADMM has a close connection with operator-splitting methods~\cite{Douglas-Rachford-56,Peaceman1955,admm1}.
But ADMM comes back to popularity due to several works like~\cite{Combettes-Pesquet-DR-2007,Goldstein-Osher-08,Wang-Yang-Yin-Zhang-2008} and the highly influential survey paper~\cite{admm3boyd}.
Numerous modern machine learning applications are inspired by ADMM, for example,~\cite{nishihara2015general,mardani2018neural,yang2011alternating,mardani2019degrees,sun2018convolutional,sun2018distributional,NIPS2019_8955}.
Linearized ADMM and gradient-based ADMM are widely used variants of ADMM.
Linearized ADMM has been studied extensively, for example, in~\cite{Chen-Teboulle-1994, He-Liao-Han-Yang-2002,lin2011linearized, Ma-APGM-2012, xu-tucker-2015, Yang-Yuan-Linearized-trace-norm, yang2011alternating, Zhang-Burger-Bresson-Osher-09, ouyang2015accelerated}; gradient-based ADMM has been studied extensively too, for example, in~\cite{Condat-2013,Vu-2013,Davis-Yin-3OS,Ma-Zhang-EGADM-2013}.
In this work, we present a general formulation for relaxed, linearized and gradient-based ADMM, and extended all of them to stochastic optimization.

Several important works studied the relaxation scheme of ADMM~\cite{eckstein1992douglas,eckstein1998operator} and propose to choose $\alpha\in (0,2)$ with the empirical suggestion for over relaxation $1<\alpha<2$.
In this work, we give a simple and elegant proof of the theoretical reason that $\alpha \in(0,2)$.
The recent work in~\cite{francca2018admm,francca2023nonsmooth} establishes the first deterministic continuous-time model for standard ADMM in the form of ordinary differential equation (ODE) for the smooth ADMM and~\cite{francca2023nonsmooth,pmlr-v97-yuan19c} extends its work to non-smooth ADMM, using the tool of differential inclusion, which is motivated by~\cite{vassilis2018differential,osher2016sparse}.
The use of stochastic and online techniques for ADMM has recently drawn a lot of interest.
\cite{WangBanerjeeIMCL2012} first proposed the online ADMM, which learns from only one sample (or a small mini-batch) at a time.
\cite{pmlr-v28-ouyang13,suzuki2013dual} proposed the variants of stochastic ADMM to attack the difficult nonlinear optimization problem inherent in $f(x,\xi)$ by linearization.
Very recently, further accelerated algorithms for the stochastic ADMM have been developed in~\cite{pmlr-v32-zhong14,pmlr-v97-huang19a}.
There are also recent works on combining variance reduction technique and stochastic ADMM
\cite{ADMM-Kwok-2016,VRSADMM-Cheng-2017,YuHuang2017}

The work in~\cite{su2016differential} is one seminal work of using the continuous-time dynamical system to analyze discrete algorithms for optimization such as Nesterov's accelerated gradient method, with important extension to high resolution and symplectic structure in~\cite{attouch2018fast,shi2022understanding,shi2019acceleration,jordan2018dynamical,francca2021dissipative}, to variational perspective in~\cite{wilson2021lyapunov,wibisono2016variational}.
The mathematical analysis of continuous formulation for stochastic optimization algorithms has become an important trend in recent years.
We select an under-represented list out of numerous works: analyzing SGD from the perspective of Langevin MCMC~\cite{zhang2017hitting,cheng2018underdamped,gao2022global,shi2023learning}, analyzing stochastic mirror-descent~\cite{zhou2017stochastic}, analyzing the SGD for the online algorithm of specific statistical models~\cite{li2017diffusion,li2016online,LTE2017SME}.
The mathematical connection between the stochastic gradient descent (SGD) and stochastic modified equation (SME) has been insightfully discovered in~\cite{LTE2017SME,LTE2019JMLR}.
This SME technique, originally arising from the numerical analysis of SDE~\cite{Milstein1995book,KPSDEbook2011}, has become the major mathematical tool for stochastic or online algorithms.
The idea of using optimal control for stochastic continuous models to find adaptive parameters like step-size for stochastic optimization is used in~\cite{LTE2017SME} to control adaptive step size and momentum, and in~\cite{an2018pid} to improve the control on momentum, in~\cite{hu2019diffusion,feng2020uniform} to choose batch size via diffusion approximation.

\section{Weak approximation to stochastic ADMM}\label{sec3}
In this section, we show the weak approximation to the generalized family of stochastic ADMM variants (\emph{G-sADMM})~\eqref{eq1.6a}---\eqref{eq1.6c}.
Throughout this section, we define $\epsilon:=\rho^{-1}$ to serve as an analog of stepsize for (stochastic) ADMM.
The $G$-sADMM scheme~\eqref{eq1.6a}---\eqref{eq1.6c} contains the relaxation parameter $\alpha$, the parameter $c=\tau / \rho$ and the implicitness parameters $\omega, \omega_{1}$.
This scheme~\eqref{eq1.6a}---\eqref{eq1.6c} is very general and includes many existing variants as follows.
It recovers all deterministic variants of ADMM for the setting $f(x, \xi) \equiv f(x)$.
If $\omega_{1}=\omega=\tau=0$, then~\eqref{eq1.6a}---\eqref{eq1.6c} is the standard stochastic ADMM (\emph{sADMM}) or online ADMM~\cite{pmlr-v28-ouyang13,pmlr-v32-zhong14}.
For the setting of parameters $\omega_{1}=0, \omega=1$ and $c>0$ it becomes stochastic version of the linearized ADMM for acceleration~\cite{He-Liao-Han-Yang-2002,Zhang-Burger-Bresson-Osher-09,ouyang2015accelerated}.
If $\omega_{1}=\omega=1$ and $c>0$, this scheme is the stochastic version of the gradient-based ADMM~\cite{Condat-2013,Vu-2013,Davis-Yin-3OS,Ma-Zhang-EGADM-2013}.
In addition, the stochastic ADMM considered in~\cite{pmlr-v28-ouyang13} for non-smooth function corresponds to the setting here with $\alpha=1$, $\omega_{1}=1$, $\omega=0$ and $\tau=\tau_{k} \propto \sqrt{k}$.

\subsection{Stochastic differential equations, weak approximation, and stochastic modified equations}\label{sec3.1}
In this subsection, we introduce several basic concepts, including the definition of stochastic differential equations, the weak approximation, and the stochastic modified equations.

Let $W_{t}$ be the standard Wiener process (i.e.~Brownian motion) in $\mathbb{R}^{d}$.
An equation of the form
\begin{equation}\label{eq3.1}
d X_{t}
=
b(X_{t}, t) d t+\sigma(X_{t}, t) d W_{t}
\end{equation}
where $b(x, t): \mathbb{R}^{d+1} \rightarrow \mathbb{R}^{d}$ and $\sigma(x, t): \mathbb{R}^{d+1} \rightarrow \mathbb{R}^{d \times d}$ are given mappings and $X_{t}$ is an unknown stochastic process, is called a stochastic differential equation (SDE).
In~\eqref{eq3.1}, vector field $b$ is the \emph{drift} coefficient, matrix $\sigma$ is the \emph{diffusion} coefficient, and $\Sigma:=\sigma \sigma^{\top}$ is the \emph{diffusion's covariance matrix}.
$X_{t}$ in~\eqref{eq3.1} on any interval $[0, T]$ can be heuristically regarded as the limit of the stochastic differential equation $
x_{n+1}
=
x_{n}+b(x_{n}, t_{n}) h+\sigma(x_{n}, t_{n})\left(W_{t_{n+1}}-W_{t_{n}}\right)
$ with $t_{n}=n h$ as $h \rightarrow 0$ and $n h\le T$.

On the other hand, for any Markov process $X_{t}$ with a continuous path, if the limits $
b(x, t)
:=
\lim _{h \downarrow 0} \frac{1}{h} \mathbb{E}\left[X_{t+h}-X_{t} \mid X_{t}=x\right]
$ and $
\Sigma(x, t)
:=
\lim_{h \downarrow 0} \frac{1}{h} \mathbb{E}\left[(X_{t+h}-X_{t})(X_{t+h}-X_{t})^{\top} \mid X_{t}=x\right]
$ exist, then the Markov process $X_{t}$ satisfies an SDE given by~\eqref{eq3.1} as its weak solution.
The weak solution roughly refers to the distributions sense in contrast to the realiaations of the path.
We next state the weak convergence below.
For any $p \geq 1$, if a uniform bound on the $p$-th moments holds, then the weak convergence is equivalent to the convergence in the Wasserstein-$p$ metric.

\begin{definition}[Weak convergence]\label{defi3.1}
We say the family (parametrized by $\epsilon$) of the stochastic sequence $\{x_{k}^{\epsilon}: k \geq 1\}$, $\epsilon>0$, weakly converges to (or is a weak approximation to) a family of continuous-time Ito processes $\{X_{t}^{\epsilon}: t \in \mathbb{R}^{+}\}$ with the order $p$ if they satisfy the following condition: For any time interval $T>0$ and for any test function $\varphi$ such that $\varphi$ and its partial derivatives up to order $2 p+2$ belong to $\mathcal{F}$, there exists a constant $C>0$ and $\epsilon_{0}>0$ such that for any $\epsilon<\epsilon_{0}$,
\begin{equation}\label{eq3.2}
\max _{1\le k\le \lfloor T / \epsilon\rfloor}\left|
\mathbb{E} \varphi(X_{k \epsilon}^{\epsilon}) - \mathbb{E} \varphi(x_{k}^{\epsilon})
\right|
\le
C \epsilon^{p}
\end{equation}
\end{definition}
The constant $C$ and $\epsilon_{0}$ in the above inequality are independent of $\epsilon$ but may depend on $T$ and $\varphi$.
In the classic applications to the numerical method for SDE~\cite{Milstein1995book}, the continuous solution $X^{\epsilon}$ is usually independent of $\epsilon$; for the stochastic modified equation, $X^{\epsilon}$ does depend on $\epsilon$.
We drop the subscript $\epsilon$ in $x_{k}^{\epsilon}$ and $X_{t}^{\epsilon}$ for notational ease whenever there is no ambiguity.

The idea of using the weak approximation and the stochastic modified equation was rigorously carried out in~\cite{LTE2017SME}, which is based on an important theorem due to~\cite{Milstein1986}.
In brief, Milstein's theorem links the one-step difference, which has been detailed above, to the global approximation in the weak sense, by checking three conditions on the momentums of the one-step difference.
Since we only consider the first-order weak approximation, Milstein's theorem is introduced in a simplified form below for only $p=1$.
The more general situations can be found in Theorem 5 in~\cite{Milstein1986}, Theorem 9.1 in~\cite{Milstein1995book} and Theorem 14.5.2 in~\cite{KPSDEbook2011}.

Let the stochastic sequence $\{x_{k}\}$ be recursively defined by the iteration written in the form associated with a function $\mathcal{A}(\cdot, \cdot, \cdot)$:
\begin{equation}\label{eq3.3}
x_{k+1}=x_{k}+\epsilon \mathcal{A}(\epsilon, x_{k}, \xi_{k+1})
, \quad 
k \geq 0
\end{equation}
where $\left\{\xi_{k}: k \geq 1\right\}$ are i.i.d.~random variables.
The initial $x_{0}=x \in \mathbb{R}^{d}$, and we define the one step difference $\bar{\Delta}=x_{1}-x$.
We use the parenthetical subscript to denote the dimensional components of a vector like $\bar{\Delta}=(\bar{\Delta}_{(i)}, 1\le i\le d)$.
Assume that there exists a function $K_{1}(x) \in \mathcal{F}$ such that $\bar{\Delta}$ satisfies the bounds of the fourth momentum
\begin{equation}\label{eq3.4}
\left|
\mathbb{E}(\bar{\Delta}_{(i)} \bar{\Delta}_{(j)} \bar{\Delta}_{(m)} \bar{\Delta}_{(l)})
\right|
\le
K_{1}(x) \epsilon^{3}
\end{equation}
for any component indices $i, j, m, l \in\{1,2, \ldots, d\}$ and any $x \in \mathbb{R}^{d}$, For any arbitrary $\epsilon>0$, consider the family of the Ito processes $X_{t}^{\epsilon}$ defined by a stochastic differential equation whose noise depends on the parameter $\epsilon$,
\begin{equation}\label{eq3.5}
d X_{t}
=
b(X_{t}) d t+\sqrt{\epsilon} \sigma(X_{t}) d W_{t}
\end{equation}
$W_{t}$ is the standard Wiener process in $\mathbb{R}^{d}$.
The initial is $X_{0}=x_{0}=x$.
The coefficient functions $b$ and $\sigma$ satisfy certain standard Lipschitz and linear growth conditions; see~\cite{Milstein1995book,LTE2019JMLR}.
Define the one step difference $\Delta=X_{\epsilon}-x$ for the SDE~\eqref{eq3.5}.

\begin{theorem}[Milstein's weak convergence theorem]\label{theo3.1}
If there exists a constant $K_{0}$ and a function $K_{2}(x) \in \mathcal{F}$, such that the following conditions of the first three moments on the error $\Delta-\bar{\Delta}$:
\begin{subequations}
\begin{align}
&
\left|\mathbb{E}(X_{\epsilon}-X_{1})\right|
\le
K_{0} \epsilon^{2}
\label{eq3.6a}
\\&
\left|\mathbb{E}(\Delta_{(i)} \Delta_{(j)}) - \mathbb{E}(\bar{\Delta}_{(i)} \bar{\Delta}_{(j)})\right|
\le
K_{1}(x) \epsilon^{2}
\label{eq3.6b}
\\&
\left|\mathbb{E}(\Delta_{(i)} \Delta_{(j)} \Delta_{(l)}) - \mathbb{E}(\bar{\Delta}_{(i)} \bar{\Delta}_{(j)} \bar{\Delta}_{(l)})\right|
\le
K_{2}(x) \epsilon^{2}
\label{eq3.6c}
\end{align}\end{subequations}
hold for any $i, j, l\in\{1,2, \ldots, d\}$ and any $x \in \mathbb{R}^{d}$, then $\{x_{k}\}$ weakly converges to $\{X_{t}\}$ with the order 1.
In light of the above theorem, we will now call~\eqref{eq3.5} the stochastic modified equation (SME) of the iterative scheme~\eqref{eq3.3}.
\end{theorem}

For the SDE~\eqref{eq3.5} at the small noise $\epsilon$, by the Ito-Taylor expansion, it is wellknown that $\mathbb{E} \Delta=b(x) \epsilon+\mathcal{O}(\epsilon^{2})$ and $\mathbb{E}\left[\Delta \Delta^{\top}\right]=\left(b(x) b(x)^{\top}+\sigma(x) \sigma(x)^{\top}\right) \epsilon^{2}+\mathcal{O}(\epsilon^{3})$ and $\mathbb{E}\left(\Pi_{m=1}^{s} \Delta_{\left(i_{m}\right)}\right)=\mathcal{O}(\epsilon^{3})$ for all integer $s \geq 3$ and the component index $i_{m}=1, \ldots, d$.
For more we refer to~\cite{KPSDEbook2011} and Lemma~1 in~\cite{LTE2017SME}.
Hence the main receipt to apply Milstein's theorem is to examine the conditions of the momentums for the discrete sequence $\bar{\Delta}=x_{1}-x_{0}$.

One prominent work~\cite{LTE2017SME} is to use the SME as a weak approximation to understand the dynamical behavior of the stochastic gradient descent (SGD).
The prominent advantage of this technique is that the fluctuation in the SGD iteration can be well captured by the fluctuation in the SME.
We brief the result as follows.
For the expectation minimization problem $\min _{x \in \mathbb{R}} f(x)=\mathbb{E}_{\xi} f(x, \xi)$, the SGD iteration is $x_{k+1}=x_{k}-\epsilon f^{\prime}\left(x_{k}, \xi_{k+1}\right)$ with the step size $\epsilon$.
Then by Theorem~\ref{theo3.1}, the corresponding SME of first-order approximation is
\begin{equation}\label{eq3.7}
d X_{t}
=
-f^{\prime}(x) d t+\sqrt{\epsilon} \sigma(x) d W_{t}
\end{equation}
with $
\sigma(x)
=
\operatorname{std}_{\xi}\left(f^{\prime}(x, \xi)\right)
=
\left(\mathbb{E}(f^{\prime}(x)-f^{\prime}(x, \xi))^{2}\right)^{1 / 2}
$.
Details can be found in~\cite{LTE2017SME}.
The SGD here is analogous to the forward-time Euler-Maruyama approximation since $\mathcal{A}(\epsilon, x, \xi)=f^{\prime}(x, \xi)$.

\subsection{Main results}\label{sec3.2}
We introduce some notations and assumptions.
We use $\|\cdot\|$ to denote the Euclidean two norm if the subscript is not specified.
All vectors are referred to as column vectors.
$f^{\prime}(x, \xi), g^{\prime}(z)$ and $f^{\prime \prime}(x, \xi), g^{\prime \prime}(z)$ refer to the first (gradient) and second (Hessian) derivatives w.r.t.~$x$.
The first assumption is \emph{Assumption I}: $f(x), g$ and for each $\xi, f(x, \xi)$, are closed proper convex functions; $A$ has full column rank.

Let $\mathcal{F}$ be the set of functions of at most polynomial growth.
To apply the SME theory, we need the following \emph{Assumption II} on the smoothness:

\begin{enumerate}
\item[(i)]
The second order derivative $f^{\prime \prime}, g^{\prime \prime}$ are uniformly bounded in $x$, and almost surely in $\xi$ for $f^{\prime \prime}(x, \xi)$; $\mathbb{E}\left\|f^{\prime}(x, \xi)\right\|_{2}^{2}$ is uniformly bounded in $x$;
\item[(ii)]
$f(x), f(x, \xi), g(x)$ and the partial derivatives up to order 5 belong to $\mathcal{F}$;
\item[(iii)]
$f^{\prime}(x)$ and $f^{\prime}(x, \xi)$ satisfy a uniform growth condition: $\left\|f^{\prime}(x)\right\|+\left\|f^{\prime}(x, \xi)\right\| \leq C(1+$ $\|x\|)$ for a constant $C$ independent of $\xi$.
\end{enumerate}
The conditions (ii) and (iii) are inherited from~\cite{Milstein1986}, which can be relaxed to a certain extent without compromising the conclusion here.
For instance, the linear growth condition of the gradient in (iii) is only needed on a compact domain.
Refer to~\cite{LTE2019JMLR} for technical and rigorous improvement for these assumptions on $f$.

With the prior notations and assumptions in hand, we present our first main theorem as follows.
Given the noisy gradient $f^{\prime}(x, \xi)$ and its expectation $f^{\prime}(x)=\mathbb{E}_{\xi} f(x, \xi)$, we define the following covariance matrix
\begin{equation}\label{eq3.8}
\Sigma(x)
=
\sigma(x) \sigma(x)^{\top}
=
\mathbb{E}_{\xi}\left[
(f^{\prime}(x, \xi)-f^{\prime}(x))
(f^{\prime}(x, \xi)-f^{\prime}(x))^{\top}
\right]
\end{equation}
We conclude

\begin{theorem}[SME for G-sADMM]\label{theo3.2}
Let $\alpha \in(0,2)$, $\omega_{1}, \omega \in\{0,1\}$ and $c=\tau / \rho \geq 0$.
Let $\epsilon=\rho^{-1} \in(0,1)$.
$\{x_{k}\}$ denote the sequence of stochastic ADMM~\eqref{eq1.6a}---\eqref{eq1.6c} with the initial choice $z_{0}=A x_{0}$.
Define $X_{t}$ as a stochastic process satisfying the SDE
\begin{equation}\label{eq3.9}
\widehat{M} d X_{t}=-\nabla V(X_{t}) d t+\sqrt{\epsilon} \sigma(X_{t}) d W_{t}
\end{equation}
where the matrix
\begin{equation}\label{eq3.10}
\widehat{M}:=c I+\left(\frac{1}{\alpha}-\omega\right) A^{\top} A
\end{equation}
Then we have $x_{k} \rightarrow X_{k \epsilon}$ in weak convergence of order 1, with the following precise meaning: for any time interval $T>0$ and for any test function $\varphi$ such that $\varphi$ and its partial derivatives up to order 4 belong to $\mathcal{F}$, there exists a constant $C$ such that
\begin{equation}\label{eq3.11}
\left|\mathbb{E} \varphi(X_{k \epsilon})-\mathbb{E} \varphi(x_{k})\right|
\le
C \epsilon
\quad
\forall k \leq\lfloor T / \epsilon\rfloor
\end{equation}
\end{theorem}

\begin{remark}\label{rema3.1}
The stochastic scheme~\eqref{eq1.6a}---\eqref{eq1.6c} is the simplest form of using only one instance of the gradient $f^{\prime}(x, \xi_{k+1})$ in each iteration.
If a batch size larger than one is used, then the one instance gradient $f^{\prime}(x, \xi_{k+1})$ is replaced by the average $\frac{1}{B_{k+1}} \sum_{i=1}^{B_{k+1}} f^{\prime}(x, \xi_{k+1}^{i})$ where $B_{k+1}>1$ is the batch size and $(\xi_{k+1}^{i})$ are $B_{k+1}$ i.i.d.~samples.
Under these settings, $\Sigma$ should be multiplied by a fact $\frac{1}{B_{t}}$ where the continuous-time function $B_{t}$ is the linear interpolation of $B_{k}$ at times $t_{k}=k \epsilon$.
The stochastic modified equation~\eqref{eq3.9} is then in the following form $\widehat{M} d X_{t}=-\nabla V(X_{t}) d t+\sqrt{\frac{\epsilon}{B_{t}}} \sigma(X_{t}) d W_{t}$.
\end{remark}

The proof of the theorem (as long as the upcoming results) is delegated to the next section.
Here we highlight one useful corollary resulting from our proof of the theorem.

\begin{figure}[!tb]
\centering
\begin{subfigure}[b]{0.45\textwidth}
\centering
\includegraphics[width = \textwidth]{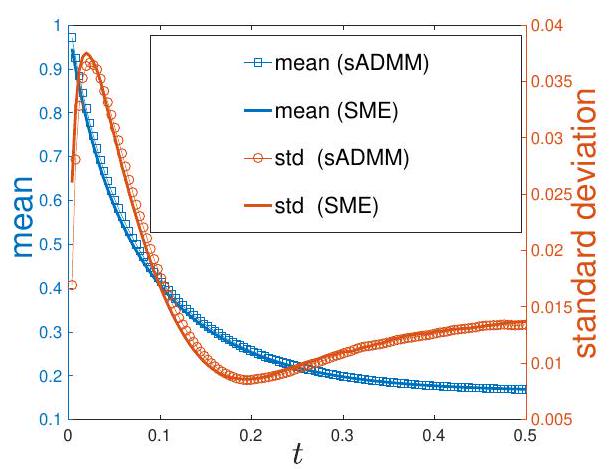}
\caption{The expectation (left axis) and standard deviation (right axis) of $x_{k}$ and $X_{t}$: $\epsilon=2^{-7}, \alpha=1.5$.}
\end{subfigure}
\hfill
\begin{subfigure}[b]{0.45\textwidth}
\centering
\includegraphics[width = \textwidth]{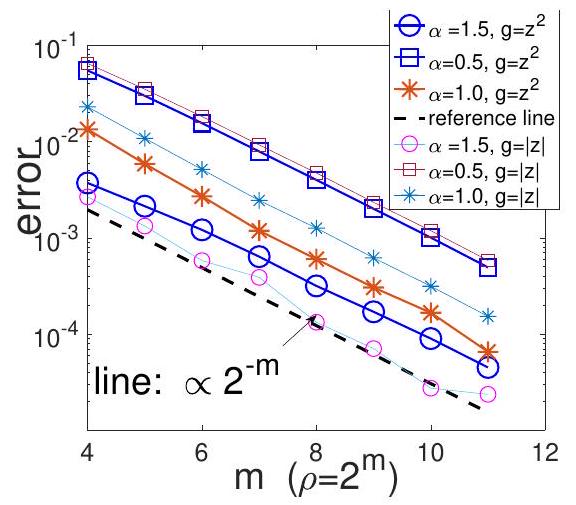}
\caption{The weak convergence error $\mathrm{err}_{m}$ versus $m$ for various $\alpha$ and $\ell_{2}, \ell_{1}$ regularization $g$.}
\end{subfigure}
\caption{
The match between the stochastic ADMM and the SME and the verification of the first-order weak approximation. The result is based on the average of $10^{5}$ independent runs: step size $\epsilon=1 / \rho=2^{-m} T$ and $T=0.5$ in (b). The details can be referred to in Section~\ref{sec5}.}
\label{fig3.1}
\end{figure}

\begin{corollary}\label{coro3.1}
For generalized stochastic ADMM, when $\rho \rightarrow \infty$, the residual $r_{k}=$ $A x_{k}-z_{k}$ satisfies the relation: $r_{k+1}=(1-\alpha) r_{k}+\mathcal{O}(\epsilon)$.
So, the necessary condition on the relaxation parameter $\alpha$ for the residual convergence $r_{k} \rightarrow 0$ as $k \rightarrow \infty$ is
$$
|1-\alpha|<1
$$
This corollary matches the empirical range $\alpha \in(0,2)$ proposed for standard ADMM with relaxation~\cite{admm3boyd}.
\end{corollary}

As an illustration of Theorem~\ref{theo3.2} for the first order weak convergence, Figure~\ref{fig3.1}b verifies this rate for various $\alpha$ and even different $g$.
To show that the SME does not only provide the expectation of the solution but also provides the fluctuation of the numerical solution $x_{k}$ for any given $\epsilon$, Figure~\ref{fig3.1}a plots the match of mean and standard deviation of $X_{t}$ from the SME versus the $x_{k}$ from the sADMM.
Section~\ref{sec5} includes more examples and tests for understanding the role of stochastic fluctuation and the relaxation parameter.

Base on~\eqref{eq3.9} in Theorem~\ref{theo3.2}, the simple asymptotic expansion for small $\epsilon$ immediately shows that the standard deviation of the stochastic ADMM $x_{k}$ is $\mathcal{O}(\sqrt{\epsilon})$ and the rescaled two standard deviations $\epsilon^{-1 / 2} \operatorname{std}(x_{k})$ and $\epsilon^{-1 / 2} \operatorname{std}(X_{k \epsilon})$ are close as the functions of the time $t_{k}=k \epsilon$.
The upcoming Section~\ref{sec3.3} discusses this issue and other conclusions on the continuous dynamics of the $z_{k}$ variable satisfying $Z_{t}=A X_{t}+\mathcal{O}(\epsilon)$.

\subsection{Continuous dynamics of $z_{k}$}\label{sec3.3}
Based on the SME~\eqref{eq3.9}, we can find the stochastic asymptotic expansion of $X_{t}^{\epsilon}$ for small $\epsilon$ in any fixed interval $[0, T]$
\begin{equation}\label{eq3.12}
X_{t}^{\epsilon} \approx X_{t}^{0}+\sqrt{\epsilon} X_{t}^{(1 / 2)}+\epsilon X_{t}^{(1)}+\ldots
\end{equation}
See Chapter 2 in~\cite{FW2012} for rigorous justification.
$X_{t}^{0}$ is \emph{deterministic} as the gradient flow of the deterministic problem: $\dot{X}_{t}^{0}=-V^{\prime}(X_{t}^{0})$, $X_{t}^{(1 / 2)}$ and $X_{t}^{(1)}$ are \emph{stochastic} and satisfy certain SDEs independent of $\epsilon$.
The useful conclusion is that the standard deviation of $X_{t}^{\epsilon}$, mainly coming from the term $\sqrt{\epsilon} X_{t}^{(1)}$, is $\mathcal{O}(\sqrt{\epsilon})$.
Hence, the standard deviation of $x_{k}$ from the stochastic ADMM is also $\mathcal{O}(\sqrt{\epsilon})$ and the rescaled two standard deviations $\epsilon^{-1 / 2} \operatorname{std}(x_{k})$ and $\epsilon^{-1 / 2} \operatorname{std}(X_{k \epsilon})$ are close as the functions of the time $t_{k}=k \epsilon$.

We can further investigate the fluctuation of the sequence $z_{k}$ generated by the stochastic ADMM and the modified equation of its continuous version $Z_{t}$.

\begin{theorem}\label{theo3.3}
We have
\begin{enumerate}
\item[(i)]
There exists a deterministic function $h(x, z)$ such that
\begin{equation}\label{eq3.13}
\dot{Z}_{t}^{\epsilon}
=
A \dot{X}_{t}^{\epsilon}+\epsilon h(X_{t}^{\epsilon}, Z_{t}^{\epsilon})
\end{equation}
such that $\{z_{k}\}$ is a weak approximation to $\{Z_{t}^{\epsilon}\}$ with the order 1, where $X_{t}^{\epsilon}$ is the solution to the SME~\eqref{eq3.9}.

\item[(ii)]
In addition, we have the following asymptotic for $Z_{t}^{\epsilon}$:
\begin{equation}\label{eq3.14}
Z_{t}^{\epsilon}
\sim
A X_{t}^{0}+\sqrt{\epsilon} A X_{t}^{(1 / 2)}+\epsilon Z_{t}^{(1)}
\end{equation}
where $Z_{t}^{(1)}$ satisfies $\dot{Z}_{t}^{(1)}=h(X_{t}^{0}, A X_{t}^{0})$ and the standard deviation of $z_{k}$ is on the order $\sqrt{\epsilon}$.
\end{enumerate}
\end{theorem}

We may consider a special case where $g(z)=\frac{1}{2}\|z\|_{2}^{2}$.
Then~\eqref{eq4.2b} becomes
$$
(1+\epsilon) z_{k+1}-z_{k}=\alpha A x_{k+1}+(-\alpha+\epsilon) z_{k}
$$
Recall the residual $r_{k}=A x_{k}-z_{k}$ and in view of~\eqref{eq4.15}, we have the following result that there exists a function $h_{1}$ such that
\begin{equation}\label{eq3.15}
\alpha R_{t}^{\epsilon}
=
(1-\alpha)\left(Z_{t}^{\epsilon}-Z_{t-\epsilon}^{\epsilon}\right)
+
\epsilon^{2} h_{1}\left(X_{t}^{\epsilon}, Z_{t}^{\epsilon}\right)
\end{equation}
and the residual $\{r_{k}\}$ is a weak approximation to $\left\{R_{t}^{\epsilon}\right\}$ with the order 1.
If $\alpha=1$ in the G-sADMM~\eqref{eq1.6a}---\eqref{eq1.6c}, then the expectation and standard deviation of $R_{t}$ and $r_{k}$ are both at order $\mathcal{O}(\epsilon^{2})$.
If $\alpha \neq 1$ in the G-sADMM~\eqref{eq1.6a}---\eqref{eq1.6c}, then the expectation and standard deviation of $R_{t}$ and $r_{k}$ are only at order $\mathcal{O}(\epsilon)$.

\section{Main theorem's and corollary's proofs}\label{sec4}
We provide the proof of our main results in this section.

\paragraph{Proof regarding SME for G-sADMM, Theorem~\ref{theo3.2}.}
The ADMM scheme is the iteration of the triplet $(x, z, \lambda)$ where $\lambda:=\epsilon u$.
In our proof, we shall show this triplet iteration can be effectively reduced to the iteration of $x$ variable only in the form of $x_{k+1}=x_{k}+\epsilon \mathcal{A}(\epsilon, x_{k}, \xi_{k+1})$.
The main approach for this reduction is the analysis of the residual $r_{k}=x_{k}-A z_{k}$ below.
This \emph{one step difference} $x_{k+1}-x_{k}$ will play a central role in the theory of weak convergence, Milstein's theorem (Theorem~\ref{theo3.1} in Section~\ref{sec3.1})~\cite{Milstein1986}.

For notational ease, we drop the random variable $\xi_{k+1}$ in the scheme~\eqref{eq1.6a}---\eqref{eq1.6c}; the readers bear in mind that $f$ and its derivatives do involve $\xi$ and all conclusions hold almost surely for $\xi$.

The optimality conditions for the scheme~\eqref{eq1.6a}---\eqref{eq1.6c} are
\begin{subequations}
\begin{align}
&
\omega_{1} \epsilon f^{\prime}(x_{k})
+
(1-\omega_{1}) \epsilon f^{\prime}(x_{k+1})
+
\epsilon A^{\top} \lambda_{k}
\notag\\&\hspace{1in}
+
A^{\top}\left(\omega A x_{k}+(1-\omega) A x_{k+1}-z_{k}\right)
+
c(x_{k+1}-x_{k})
=0
\label{eq4.1a}
\\&
\epsilon g^{\prime}(z_{k+1})=\epsilon \lambda_{k}+\alpha A x_{k+1}+(1-\alpha) z_{k}-z_{k+1}
\label{eq4.1b}
\\&
\epsilon \lambda_{k+1}=\epsilon \lambda_{k}+\alpha A x_{k+1}+(1-\alpha) z_{k}-z_{k+1}
\label{eq4.1c}
\end{align}
\end{subequations}
Note that due to~\eqref{eq4.1b} and~\eqref{eq4.1c}, the last condition~\eqref{eq4.1c} can be replaced by $\lambda_{k+1}=g^{\prime}(z_{k+1})$.
So, without loss of generality, one can assume that
$$
\lambda_{k} \equiv g^{\prime}(z_{k})
$$
for any integer $k$.
The optimality conditions~\eqref{eq4.1a}---\eqref{eq4.1c} now can be written only in the pair $(x, z)$:
\begin{subequations}
\begin{align}
&
\omega_{1} \epsilon f^{\prime}(x_{k})+(1-\omega_{1}) \epsilon f^{\prime}(x_{k+1})+\epsilon A^{\top} g^{\prime}(y_{k})
\notag\\&\hspace{1in}
+A^{\top}\left(\omega A x_{k}+(1-\omega) A x_{k+1}-z_{k}\right)+c(x_{k+1}-x_{k})=0
\label{eq4.2a}
\\
&
\epsilon g^{\prime}(z_{k+1})-\epsilon g^{\prime}(z_{k})=\alpha A x_{k+1}+(1-\alpha) z_{k}-z_{k+1}
\label{eq4.2b}
\end{align}\end{subequations}
We treat $(x_{k+1}, z_{k+1})$ in~\eqref{eq4.2a}---\eqref{eq4.2b} as a function of $x_{k}, z_{k}$ and $\epsilon$ and seek the asymptotic expansion of this function in the regime of small $\epsilon$, by using Dominant Balance technique (see, e.g.~\cite{white2010asymptotic}) with the Taylor expansion for $f^{\prime}(x_{k+1})$ and $g^{\prime}(z_{k+1})$ around $x_{k}$ and $z_{k}$, respectively.
The expansion of the one-step difference up to the order $\epsilon$ is
\begin{subequations}
\begin{align}
x_{k+1}-x_{k}
&=
-M^{-1} A^{\top} r_{k}-\epsilon M^{-1}\left[f^{\prime}(x_{k})+A^{\top} g^{\prime}(z_{k})\right]-\epsilon N_{k} A^{\top} r_{k}+\epsilon^{2} c_{k}
\label{eq4.3a}
\\
z_{k+1}-z_{k}
&=
\alpha\left(I-A M^{-1} A^{\top}\right) r_{k}+\epsilon c_{k}^{\prime}
\label{eq4.3b}
\end{align}\end{subequations}
where $r_{k}$ is the \emph{residual}
\begin{equation}\label{eq4.4}
r_{k}:=A x_{k}-z_{k}
,
\end{equation}
and the matrix $M$ is
\begin{equation}\label{eq4.5}
M:=c I+(1-\omega) A^{\top} A
.
\end{equation}
$N_{k}$ in~\eqref{eq4.3a} is a matrix depending on the hessian $f^{\prime \prime}(x_{k})$ and $M$.
$c_{k}$ and $c_{k}^{\prime}$ are uniformly bounded in $k$ (and in random element $\xi$) due to the bound of functions $f^{\prime \prime}$ and $g^{\prime \prime}$ in \emph{Assumption II}.
Throughout the rest of the paper, we shall use the notation $\mathcal{O}\left(\epsilon^{p}\right)$ to denote these terms bounded by a (nonrandom) constant multiplier of $\epsilon^{p}$, for $p=1,2, \ldots$.

~\eqref{eq4.3a}---\eqref{eq4.3b} clearly shows that the dynamics of the pair $(x_{k+1}, z_{k+1})$ depends on the residual $r_{k}=A x_{k}-z_{k}$ in the previous step.
In the limit of $\epsilon=0$, the one step difference (i.e.~\emph{truncation error}) $
x_{k+1}-x_{k}
\rightarrow
-M^{-1} A^{\top} r_{k}$
 which does not vanish unless $r_{k}$ tends to zero.
We turn to the analysis of the residual.
\eqref{eq4.3b} together with~\eqref{eq4.3a} implies
$$
z_{k+1}-z_{k}=-\alpha A M^{-1} A^{\top} r_{k}+\alpha r_{k}+\mathcal{O}(\epsilon)
$$
by absorbing all $\epsilon$ order terms into $\mathcal{O}(\epsilon)$.
By using~\eqref{eq4.2b} for $A x_{k+1}$, we have the recursive relation for residual $r_{k+1}$:
\begin{equation}\label{eq4.6}
\begin{aligned}
r_{k+1}=A x_{k+1}-z_{k+1}
&=
\left(\frac{1}{\alpha}-1\right)\left(z_{k+1}-z_{k}\right)
+
\frac{\epsilon}{\alpha}\left(g^{\prime}(z_{k+1})-g^{\prime}(z_{k})\right)
\\&=
(1-\alpha)\left(I-A M^{-1} A^{\top}\right) r_{k}+\mathcal{O}(\epsilon)
\end{aligned}
\end{equation}
Since $r_{0}=0$ at the initial step by setting $z_{0}=A x_{0}$,~\eqref{eq4.6} shows that $r_{1}=A x_{1}-y_{1}$ become $\mathcal{O}(\epsilon)$ after one iteration.

Next, we have the following important property, particularly with assumption $\alpha=1$, we have a stronger result than~\eqref{eq4.6}.

\begin{proposition}\label{prop4.1}
For any integer $k \geq 0$, if $r_{k}=\mathcal{O}(\epsilon)$, then
\begin{equation}\label{eq4.7}
r_{k+1}
=
\left(1-\alpha+\epsilon \alpha g^{\prime \prime}(z_{k})\right)\left(r_{k}+A(x_{k+1}-x_{k})\right)+\mathcal{O}(\epsilon^{3})
\end{equation}
If $\alpha=1$,~\eqref{eq4.7} reduces to the second order in $\epsilon$:
\begin{equation}\label{eq4.8}
r_{k+1}
=
\epsilon \alpha g^{\prime \prime}(z_{k})\left(r_{k}+A(x_{k+1}-x_{k})\right)
+
\mathcal{O}(\epsilon^{3})
=
\mathcal{O}(\epsilon^{2})
\end{equation}
\end{proposition}

\begin{proof}[Proof of Proposition~\ref{prop4.1}]
Since $r_{k}=A x_{k}-z_{k}=\mathcal{O}(\epsilon)$, then the one step difference $x_{k+1}-x_{k}$ and $z_{k+1}-z_{k}$ are both at order $\mathcal{O}(\epsilon)$ because of~\eqref{eq4.3a} and~\eqref{eq4.3b}.
We solve $\delta z:=z_{k+1}-z_{k}$ from~\eqref{eq4.2b} by linearizing the implicit term $g^{\prime}(z_{k+1})$ and use the assumption that the third order derivative of $g$ exits:
$$
\epsilon g^{\prime \prime}(z_{k}) \delta z+\epsilon \mathcal{O}\left((\delta z)^{2}\right)+\delta z=\alpha\left(r_{k}+A \delta x\right)
$$
where $\delta x:=x_{k+1}-x_{k}=\mathcal{O}(\epsilon)$.
Then since $\mathcal{O}\left((\delta z)^{2}\right)=\mathcal{O}(\epsilon^{2})$, the expansion of $\delta z=z_{k+1}-z_{k}$ in $\epsilon$ from the above equation leads to
\begin{equation}\label{eq4.9}
z_{k+1}-z_{k}
=
\alpha\left(1-\epsilon g^{\prime \prime}(z_{k})\right)\left(r_{k}+A(x_{k+1}-x_{k})\right)
+
\mathcal{O}(\epsilon^{3})
\end{equation}
Together with the definition of $r_{k+1}$, the above suggests
$$\begin{aligned}
r_{k+1}
&=
r_{k}+A(x_{k+1}-x_{k})-(z_{k+1}-z_{k})
\\&=
\left(1-\alpha+\epsilon \alpha g^{\prime \prime}(z_{k})\right)\left(r_{k}+A(x_{k+1}-x_{k})\right)
+
\mathcal{O}(\epsilon^{3})
\end{aligned}$$
concluding the proposition.
\end{proof}

We now focus on~\eqref{eq4.3a} for $x_{k+1}-x_{k}$, which can be rewritten as
\begin{equation}\label{eq4.10}
\begin{aligned}
M(x_{k+1}-x_{k})
=&
-
A^{\top} r_{k}
-
\epsilon\left(
f^{\prime}(x_{k})+A^{\top} g^{\prime}(z_{k})
\right)
\\&
+
\epsilon^{2}(1-\omega_{1}) f^{\prime \prime}(x_{k}) M^{-1}\left(
f^{\prime}(x_{k})+A^{\top} g^{\prime}(z_{k})+\frac{1}{\epsilon} A^{\top} r_{k}
\right)
+
\mathcal{O}(\epsilon^{3})
\end{aligned}\end{equation}
This expression does not contain the parameter $\alpha$ explicitly, but the residual $r_{k}=A x_{k}-$ $y_{k}$ significantly depends on $\alpha$ (see Proposition~\ref{prop4.1}).
If $\alpha=1$, then $r_{k}$ is on the order of $\epsilon^{2}$, which leads to the useful fact that there is no contribution from $r_{k}$ to the weak approximation of $x_{k}$ at the order 1.
But for the relaxation case where $\alpha \neq 1$, $r_{k}$ contains the first order term coming from $z_{k+1}-z_{k}$.

To obtain a second order for some "residual" for the relaxed scheme where $\alpha \neq 1$, we need a new definition, $\alpha$-residual, to account for the gap induced by $\alpha$.
Motivated by~\eqref{eq4.1b}, we first define
\begin{equation}\label{eq4.11}
r_{k+1}^{\alpha}
:=
\alpha A x_{k+1}+(1-\alpha) z_{k}-z_{k+1}
\end{equation}
It is connected to the original residual $r_{k+1}$ and $r_{k}$ since it is easy to check that
\begin{equation}\label{eq4.12}
\begin{aligned}
r_{k+1}^{\alpha}
&=
\alpha r_{k+1}+(\alpha-1)(z_{k+1}-z_{k})
\\&=
\alpha r_{k}+\alpha A(x_{k+1}-x_{k})-(z_{k+1}-z_{k})
\end{aligned}\end{equation}
Obviously, when $\alpha=1$, this $\alpha$-residual $r^{\alpha}$ is the original residual $r=A x-y$.
In our proof, we need a modified $\alpha$-residual, denoted by
\begin{equation}\label{eq4.13}
\widehat{r}_{k+1}^{\alpha}
:=
\alpha r_{k}+(\alpha-1)(z_{k+1}-z_{k})
\end{equation}

\begin{figure}[!tb]
\centering
\includegraphics[width = 0.6\textwidth]{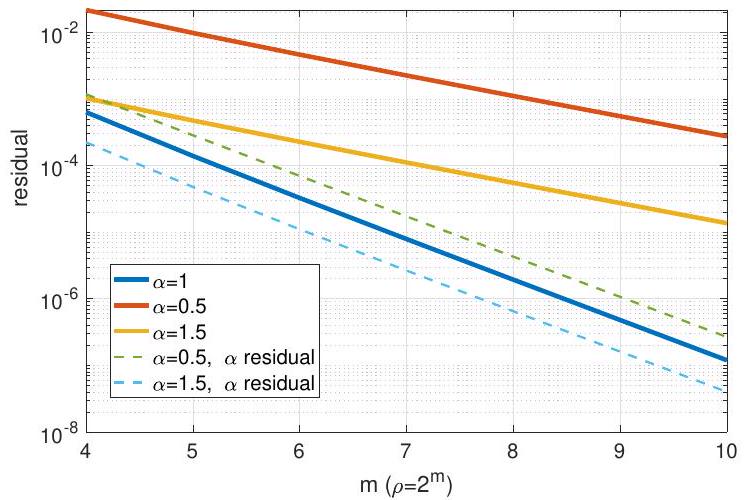}
\caption{
Illustration of the order of the residual $r_{k}$ and the $\alpha$-residual $r_{k}^{\alpha}$ defined in~\eqref{eq4.11}.}
\label{fig4.1}
\end{figure}

We shall show that both $r_{k+1}^{\alpha}$ and $\widehat{r}_{k+1}^{\alpha}$ are as small as $\mathcal{O}(\epsilon^{2})$:
$$
r_{k+1}^{\alpha}=\mathcal{O}(\epsilon^{2})
\quad \text { and } \quad
\widehat{r}_{k+1}^{\alpha}=\mathcal{O}(\epsilon^{2})
$$
In fact, by the second equality of~\eqref{eq4.12}, \eqref{eq4.9} becomes $z_{k+1}-z_{k}=\left(1-\epsilon g^{\prime \prime}(z_{k})\right)\left(r_{k+1}^{\alpha}+\right.$ $\left.z_{k+1}-z_{k}\right)+\mathcal{O}(\epsilon^{3})$, so $r_{k+1}^{\alpha}=\epsilon\left(1+\epsilon g^{\prime \prime}(z_{k})\right) g^{\prime \prime}(z_{k})\left(z_{k+1}-z_{k}\right)+\mathcal{O}(\epsilon^{3})$ which is $\mathcal{O}(\epsilon^{2})$ since $z_{k+1}-z_{k}=\mathcal{O}(\epsilon)$.
The difference between $\left(z_{k+1}-z_{k}\right)$ and $\left(z_{k+2}-z_{k+1}\right)$, is at the order $\epsilon^{2}$ due to truncation error of the central difference scheme.
Then we have the conclusion
\begin{equation}\label{eq4.14}
\widehat{r}_{k+1}^{\alpha}
=
\alpha r_{k}+(\alpha-1)\left(z_{k+1}-z_{k}\right)=\mathcal{O}(\epsilon^{2})
\end{equation}
by shifting the subscript $k$ by one.
By~\eqref{eq4.13} and the above proposition, we have $r_{k}=\left(\frac{1}{\alpha}-1\right)\left(z_{k+1}-z_{k}\right)+\mathcal{O}(\epsilon^{2})$.
Furthermore, due to~\eqref{eq4.9}, $
r_{k}
=
\left(\frac{1}{\alpha}-1\right)(z_{k+1}-z_{k})
+
\mathcal{O}(\epsilon^{2})
=
(1-\alpha)\left(r_{k}+A(x_{k+1}-x_{k})\right)
+
\mathcal{O}(\epsilon^{2})
$ which gives
\begin{equation}\label{eq4.15}
r_{k}
=
\left(\frac{1}{\alpha}-1\right)(z_{k+1}-z_{k})
+
\mathcal{O}(\epsilon^{2})
=
\left(\frac{1}{\alpha}-1\right) A(x_{k+1}-x_{k})
+
\mathcal{O}(\epsilon^{2})
\end{equation}
and it follows $z_{k+1}-z_{k}=A(x_{k+1}-x_{k})
+
\mathcal{O}(\epsilon^{2})
$.

\begin{remark}\label{rema4.1}
As an illustration, the figure above for a toy example (see numerical example part in Section~\ref{sec5}) shows that $r_{k} \sim \mathcal{O}(\epsilon)$ for $\alpha \neq 1$ while $r_{k} \sim \mathcal{O}(\epsilon^{2})$ for $\alpha=1$ and the $\alpha$-residual $r_{k}^{\alpha}$ is $\mathcal{O}(\epsilon^{2})$ regardless of $\alpha$ (the solid lines are $r_{k}$ and the dashed lines are $r_{k}^{\alpha}$).
\end{remark}

With the above preparations for the residual analysis, we now apply Theorem~\ref{theo3.1} to show the main conclusion of our theorem.
Combining~\eqref{eq4.10} and~\eqref{eq4.15}, and noting the Taylor expansion of $g^{\prime}(z_{k}): g^{\prime}(y_{k})=g^{\prime}\left(A x_{k}-r_{k}\right)=g^{\prime}\left(A x_{k}\right)+\mathcal{O}(\epsilon)$ since $r_{k}=\mathcal{O}(\epsilon)$, and now putting back the random element $\xi$ into $f^{\prime}$, we have
\begin{equation}\label{eq4.16}
\begin{aligned}
M(x_{k+1}-x_{k})
=&
-
\epsilon\left(f^{\prime}\left(x_{k}, \xi_{k+1}\right)
+
A^{\top} g^{\prime}\left(A x_{k}\right)\right)
\\&
-
\left(\frac{1}{\alpha}-1\right) A^{\top} A(x_{k+1}-x_{k})
+
\mathcal{O}(\epsilon^{2})
\end{aligned}\end{equation}
For convenience, introduce the matrix
\begin{equation}\label{eq4.17}
\widehat{M}
:=
M+\frac{1-\alpha}{\alpha} A^{\top} A=c+\left(\frac{1}{\alpha}-\omega\right) A^{\top} A
\end{equation}
and let $\widehat{x}_{k}:=\widehat{M} x_{k}$, $\delta \widehat{x}_{k+1}=\widehat{M}(x_{k+1}-x_{k})$.
Then
$$
\delta \widehat{x}
=
-
\epsilon V^{\prime}(x, \xi)
+
\epsilon^{2}\left((1-\omega_{1}) f^{\prime \prime} M^{-1} V^{\prime}(x)-A^{\top} \theta\right)
+
\mathcal{O}(\epsilon^{3})
$$
The final step is to compute the momentums and verify the conditions in Milstein's theorem (Theorem~\ref{theo3.1}) on these momentums.
We have
\begin{enumerate}
\item[(i)]
\begin{equation}\label{eq4.18}
\mathbb{E}[\delta \widehat{x}]
=
-\epsilon \mathbb{E} V^{\prime}(x, \xi)+\mathcal{O}(\epsilon^{2})=-\epsilon V^{\prime}(x)+\mathcal{O}(\epsilon^{2})
\end{equation}
\item[(ii)]
$$\begin{aligned}
\mathbb{E}\left[\delta \widehat{x} \delta \widehat{x}^{\top}\right]
&=
\epsilon^{2} \mathbb{E}\left[
(f^{\prime}(x, \xi)+A^{\top} g^{\prime}(x))
(f^{\prime}(x, \xi)+A^{\top} g^{\prime}(x))^{\top}
\right]
+\mathcal{O}(\epsilon^{3})
\\&=
\epsilon^{2}(V^{\prime}(x) V^{\prime}(x)^{\top})
+
\epsilon^{2} \mathbb{E}\left[
(f^{\prime}(x, \xi)-f^{\prime}(x))
(f^{\prime}(x, \xi)-f^{\prime}(x))^{\top}
\right]
+
\mathcal{O}(\epsilon^{3})
\end{aligned}$$
\item[(iii)]
It is trivial that $
\mathbb{E}\left[\Pi_{j=1}^{s} \delta x_{i_{j}}\right]
=
\mathcal{O}(\epsilon^{3})
$ for $s \geq 3$ and $i_{j}=1, \ldots, d$.
\end{enumerate}
This in all completes the proof of Theorem~\ref{theo3.2}.

\hfill\qed

\begin{remark}\label{rema4.2}
We do not pursue the second-order weak approximation as for the SGD scheme, due to the complicated issue of the residuals.
In addition, the proof requires a regularity condition for the functions $f$ and $g$.
So, our theoretic theorems can not cover the non-smooth function $g$.
Our numerical tests in Section~\ref{sec5} suggest that the conclusion holds too for $l_{1}$ regularization function $g(z)=\|z\|_{1}$.
\end{remark}

\begin{remark}\label{rema4.3}
In general applications, it can be difficult to obtain the explicit expression of the covariance matrix $\Sigma(x)$ as a function of $x$, except in very few simplified cases.
In applications of empirical risk minimization, the function $f$ is the empirical average of the loss on each sample $f_{i}: f(x)=\frac{1}{N} \sum_{i=1}^{N} f_{i}(x)$.
The diffusion's covariance matrix $\Sigma(x)$ in~\eqref{eq3.8} becomes the following form
\begin{equation}\label{eq4.19}
\Sigma_{N}(x)
=
\frac{1}{N} \sum_{i=1}^{N}
\left(f^{\prime}(x)-f_{i}^{\prime}(x)\right)
\left(f^{\prime}(x)-f_{i}^{\prime}(x)\right)^{\top}
\end{equation}
It is clear that if $f_{i}(x)=f(x, \xi_{i})$ with $N$ i.i.d.~samples $\xi_{i}$, then $\Sigma_{N}(x) \rightarrow \Sigma(x)$ as $N \rightarrow \infty$.
\end{remark}

\paragraph{Proof of Corollary~\ref{coro3.1}}
Based on Proposition~\ref{prop4.1},
$$
r_{k+1}
=
(1-\alpha)\left(r_{k}+A(x_{k+1}-x_{k})\right)
+
\epsilon \alpha g^{\prime \prime}(z_{k})\left(r_{k}+A(x_{k+1}-x_{k})\right)
+
\mathcal{O}(\epsilon^{3})
$$
Since $g^{\prime \prime}$ are bounded and $x_{k+1}-x_{k}=\mathcal{O}(\epsilon)$, then $r_{k+1}=(1-\alpha) r_{k}+\mathcal{O}(\epsilon)$, with the leading term $(1-\alpha) r_{k}$.
Therefore when $k \rightarrow \infty$, we need $|1-\alpha|<1$ for $r_{k}$ to converge to zero.
This completes the proof.

\hfill\qed

\paragraph{Proof of Theorem~\ref{theo3.3}}
~\eqref{eq4.15} and~\eqref{eq4.3b} indicate that
$$
z_{k+1}-z_{k}=A(x_{k+1}-x_{k})+\mathcal{O}(\epsilon^{2})
$$
Therefore $\dot{Z}_{t}^{\epsilon}=A X_{t}^{\epsilon}+\epsilon$.
The function $h$ is deterministic because the update of $z_{k+1}$ in~\eqref{eq4.2b} is deterministic conditioned on the given $x_{k+1}$.
The function $h$ depends on $g^{\prime \prime}$ and others, which is quite challenging to obtain the explicit expression, completing the proof.

\hfill\qed

\section{Numerical examples}\label{sec5}
We show some remarks first to elaborate our main theorem in practical settings before we present our illustrative examples.

In empirical risk minimization problems, the function $f$ is the empirical average of the loss on each sample $f_{i}$: $
f(x)=\frac{1}{N} \sum_{i=1}^{N} f_{i}(x)
$.
The diffusion matrix $\Sigma(x)$ in~\eqref{eq3.8} can be written as the following:
\begin{equation}\label{eq5.1}
\Sigma_{N}(x)
=
\frac{1}{N} \sum_{i=1}^{N}\left(f^{\prime}(x)-f_{i}^{\prime}(x)\right)\left(f^{\prime}(x)-f_{i}^{\prime}(x)\right)^{\top}
\end{equation}
If $f_{i}(x)=f(x, \xi_{i})$ with $N$ i.i.d.~samples $\xi_{i}$, then $\Sigma_{N}(x) \rightarrow \Sigma(x)$ as $N \rightarrow \infty$.
In addition, the stochastic scheme~\eqref{eq1.6a}---\eqref{eq1.6c} uses only one instance of the gradient $f^{\prime}(x, \xi_{k+1})$ in each iteration.
Under cases where more than one samples are computed in each iteration, then the one instance gradient $f^{\prime}(x, \xi_{k+1})$ is replaced by the batch gradient $\frac{1}{B_{k+1}} \sum_{i=1}^{B_{k+1}} f^{\prime}(x, \xi_{k+1}^{i})$ where $B_{k+1}\ge 1$ is the batch size and $(\xi_{k+1}^{i})$ are $B_{k+1}$ i.i.d.~samples.
Under these settings, $\Sigma$ should be multiplied by a factor $\frac{1}{B_{t}}$ where the continuous-time function $B_{t}$ is the linear interpolation of $B_{k}$ at times $t_{k}=k \epsilon$.
The stochastic modified equation~\eqref{eq3.9} is then of the following form
\begin{equation}\label{eq5.2}
\widehat{M} d X_{t}
=
-\nabla V(X_{t}) d t+\sqrt{\frac{\epsilon}{B_{t}}} \sigma(X_{t}) d W_{t}
\end{equation}

\subsection{A toy example}\label{sec5.1}
In~\eqref{eq1.2}, consider $f(x, \xi)=(\xi+1) x^{4}+(2+\xi) x^{2}-(1+\xi) x$, where $\xi$ is a Bernoulli random variable taking values $-1$ or +1 with equal probability.
For $g(z)$ we consider both of the cases of $g(z)=z^{2}$ and $g(z)=|z|$, and we fix the matrix $A=I$.
Recall in the stochastic ADMM~\eqref{eq1.6a}---\eqref{eq1.6c}, $c=\tau / \rho$ and $\omega \in[0,1]$.
We choose $c=\omega=1$ such that $\widehat{M}=\frac{1}{\alpha} I$.
When $g(z)=z^{2}$, the optimum solution of~\eqref{eq1.2} is $x_{*}=0.16374$ and the SME is
\begin{equation}\label{eq5.3}
\frac{1}{\alpha} d X_{t}=-\left(4 x^{3}+6 x-1\right) d t+\sqrt{\epsilon}\left|4 x^{3}+2 x-1\right| d W_{t}
\end{equation}
It is easy to verify that these settings satisfy the assumptions in our main theorem.

Throughout the analyses below, we also numerically investigated the convergence rate for the non-smooth penalty $g(z)=|z|$, even though this $\ell_{1}$ regularization function does not satisfy our assumptions in theory.
For the corresponding SDE, formally we can write
$$
\frac{1}{\alpha} d X_{t}
=
-\left(4 x^{3}+4 x-1+\operatorname{sign}(x)\right) d t+\sqrt{\epsilon}\left|4 x^{3}+2 x-1\right| d W_{t}
$$
by using the sign function for $g^{\prime}(z)$.
We only use the above SME for this simulation.
The rigorous expression requires the concept of stochastic differential inclusion, which is out of the scope of this work.

We initialize $x_{0}, z_{0}, \lambda_{0}$ as $x_{0}=z_{0}=1.0$ and $\lambda_{0}=g^{\prime}\left(z_{0}\right)$ and set the terminal time as $T=0.5$.

\paragraph{Convergence between generalized stochastic ADMM and its continuous model.}
First of all, we verify that the SME solution not only provides the approximation of the expectation of the G-sADMM~\eqref{eq1.6a}---\eqref{eq1.6c} solution, but also depicts the fluctuation of the numerical solution $x_{k}$ of~\eqref{eq1.6a}---\eqref{eq1.6c} for any given $\epsilon$.
In Figure~\ref{fig3.1}a we compare the mean and standard deviation ("std") between $x_{k}$ and $X_{k \epsilon}$ for $\epsilon=2^{-7}$ and under $g(z)=z^{2}$ as an illustration of our argument.
The left vertical axis together with two blue curves depicts the means of the two solution paths and the right vertical axis with two red curves are the standard deviations of the two solution paths respectively.
We see from Figure~\ref{fig3.1}a that both the mean curves of the SME/GsADMM and the std curves of the SME/GsADMM coincide with each other.

\begin{figure}[!tb]
\centering
\begin{subfigure}[b]{0.47\textwidth}
\centering
\includegraphics[width = \textwidth]{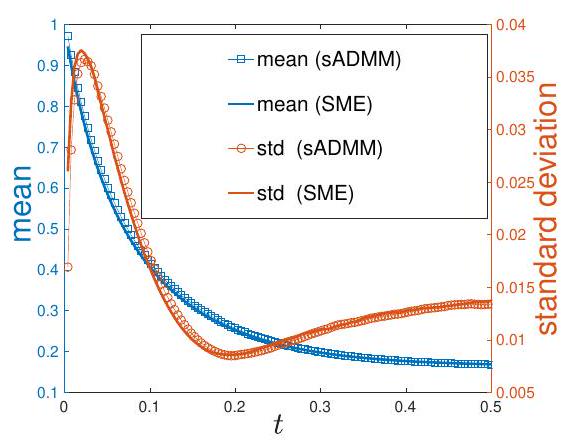}
\caption{The expectation (left axis) and standard deviation (right axis) of $x_{k}$ and $X_{t}$; $\epsilon=2^{-7}; \alpha=1.5$.}
\end{subfigure}
\hfill
\begin{subfigure}[b]{0.43\textwidth}
\centering
\includegraphics[width = \textwidth]{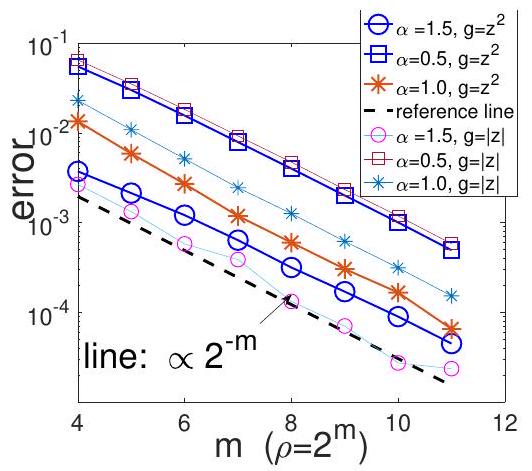}
\caption{The weak convergence error err$_{m}$ versus $m$ for various $\alpha$ and $\ell_{2}, \ell_{1}$ regularization $g$.}
\end{subfigure}
\caption{
The match between the stochastic ADMM and the SME and the verification of the first-order weak approximation. The result is based on the average of $10^{5}$ independent runs. The step size $\epsilon=1 / \rho=2^{-m} T$ and $T=0.5$ in (b). The details can be referred to in Section~\ref{sec5}.}
\label{fig5.1}
\end{figure}

\begin{figure}[!tb]
\centering
\includegraphics[width = \textwidth]{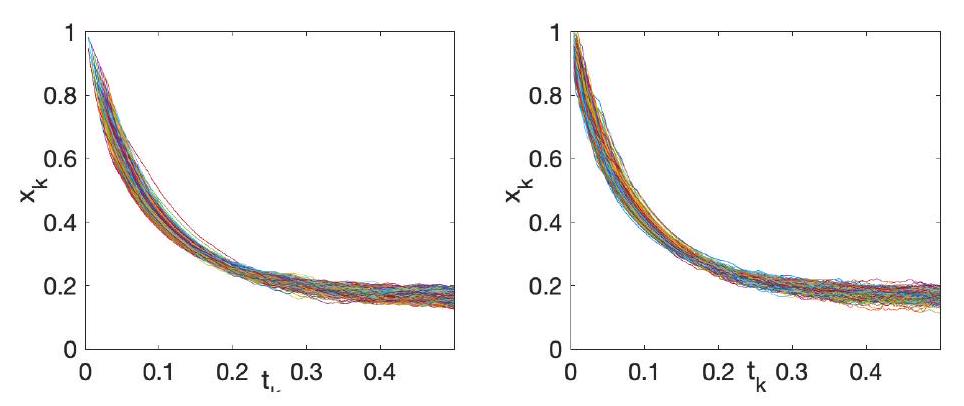}
\caption{
The 400 sample trajectories from stochastic ADMM (left) and SME (right).}
\label{fig5.2}
\end{figure}

In Figure~\ref{fig3.1}b we plot the order of weak convergence in Theorem~\ref{theo3.2}.
We choose a test function $\varphi(x)=x+x^{2}$ to evaluate the weak convergence error:
$$
\text { err }
:=
\max _{1 \leq k \leq\lfloor T / \epsilon\rfloor}
\left|\mathbb{E} \varphi(x_{k})-\mathbb{E} \varphi(X_{k \epsilon})\right|
$$
For each $m=4,5, \ldots, 11$, we set $\rho=2^{m} / T$ and $\epsilon=T 2^{-m}$.
We plot $m$ as the $x$-axis and $\operatorname{err}_{m}$ as the $y$-axis in Figure~\ref{fig3.1}b and apply the semi-log scale for both three values of relaxation parameter $\alpha$ and two cases of $g(z)$.
We see that $\operatorname{err}_{m}$ satisfies~\eqref{eq3.11} in Theorem~\ref{theo3.2} and thus verifies the weak convergence of order one.

In addition, we plot 400 sample trajectories of the GsADMM together in Figure~\ref{fig5.2}, left panel and 400 sample trajectories of the SME together in Figure~\ref{fig5.2}, right panel, with $g(z)=z^{2}$ and $\alpha=1.5$.
This further shows the match between the mean and the std between the solution paths of the stochastic ADMM and the SME.

\begin{figure}[!tb]
\centering
\includegraphics[width = 0.6\textwidth]{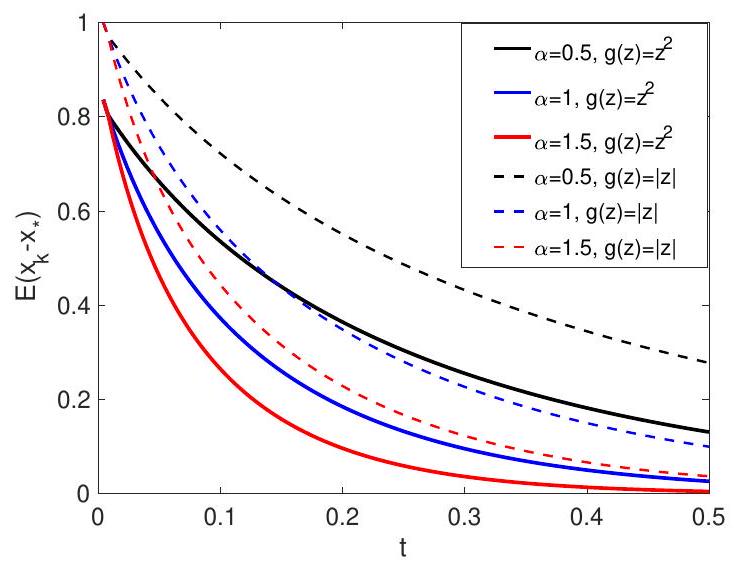}
\caption{
The expectation of the error to true minimizer $x_{k}-x_{*}$ when $\alpha \in(0,2)$ varies. The result is based on the average of 10000 runs of stochastic ADMM sequences.}
\label{fig5.3}
\end{figure}

\begin{figure}[!tb]
\includegraphics[width = \textwidth]{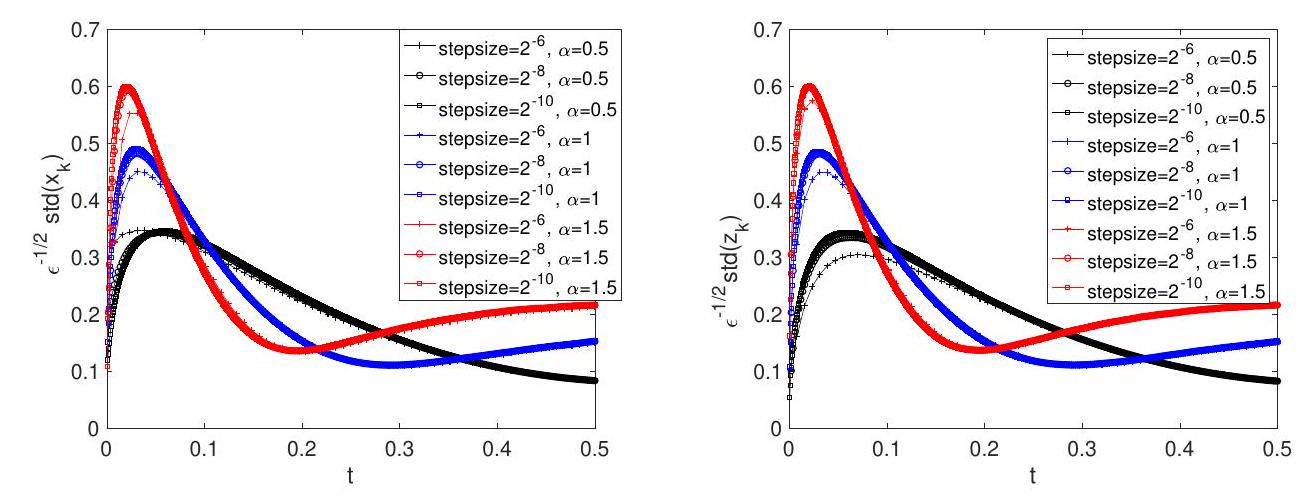}
\caption{
The rescaled std of $x_{k}$ and $z_{k}$.}
\label{fig5.4}
\end{figure}

\paragraph{The effect of $\alpha$.}
In Figure~\ref{fig5.3}, we explore how $\alpha$ affects and accelerate the optimization of the stochastic ADMM~\eqref{eq1.6a}---\eqref{eq1.6c}.
The continuous model~\eqref{eq5.3} supports this observation here that a large $\alpha$ helps fast decay toward the true minimizer.
Similar effects have been described for the deterministic ADMM in~\cite{pmlr-v97-yuan19c}.

\paragraph{The fluctuation of $x_{k}$ and $z_{k}$.}
Next, we test the orders of the standard deviation of $x_{k}$ and $z_{k}$.
Previously in Section~\ref{sec3.3} (see~\eqref{eq3.12} and~\eqref{eq3.14}), we mentioned that theoretically the standard deviation of the sequences $x_{k}$ and $z_{k}$ are of order $\sqrt{\epsilon}$.
So the two sequences $\epsilon^{-1 / 2} \operatorname{std}(x_{k})$ and $\epsilon^{-1 / 2} \operatorname{std}(z_{k})$ should only depends on $\alpha$ regardless of $\epsilon$, which is shown in Figure~\ref{fig5.4}.

\begin{figure}[!tb]
\centering
\begin{subfigure}[b]{0.32\textwidth}
\centering
\includegraphics[width = \textwidth]{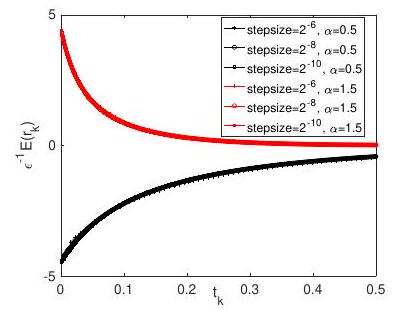}
\caption{$\alpha\ne 1$.~$g(z)=z^{2}$.}
\end{subfigure}
\hfill
\begin{subfigure}[b]{0.32\textwidth}
\centering
\includegraphics[width = \textwidth]{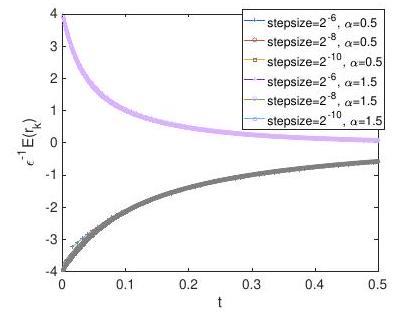}
\caption{$\alpha\ne 1$.~$g(z)=|z|$.}
\end{subfigure}
\hfill
\begin{subfigure}[b]{0.32\textwidth}
\centering
\includegraphics[width = \textwidth]{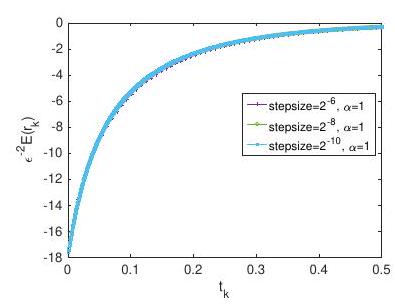}
\caption{$\alpha=1$.~$g(z)=z^{2}$.}
\end{subfigure}
\begin{subfigure}[b]{0.32\textwidth}
\centering
\includegraphics[width = \textwidth]{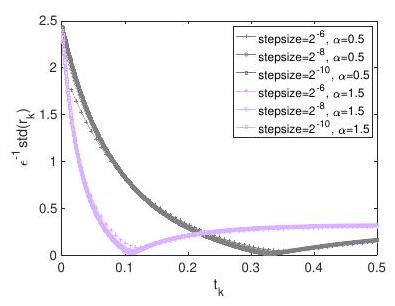}
\caption{$\alpha\ne 1$.~$g(z)=z^{2}$.}
\end{subfigure}
\hfill
\begin{subfigure}[b]{0.32\textwidth}
\centering
\includegraphics[width = \textwidth]{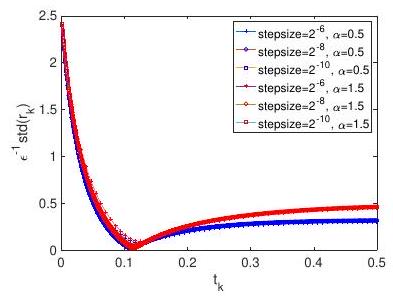}
\caption{$\alpha \neq 1$.~$g(z)=|z|$.}
\end{subfigure}
\hfill
\begin{subfigure}[b]{0.32\textwidth}
\centering
\includegraphics[width = \textwidth]{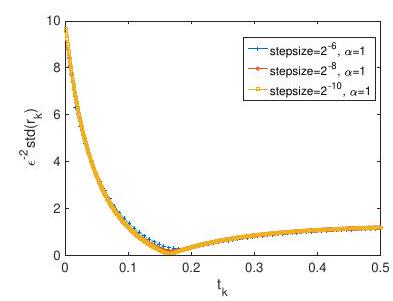}
\caption{$\alpha=1$.~$g(z)=z^{2}$.}
\end{subfigure}
\caption{
The scaling of the mean (first row) and std (second row) of residual $r_{k}$: for $\alpha \neq 1$, both are on the order $\epsilon^{-1}$ while for $\alpha=1$, both are on the order $\epsilon^{-2}$.}
\label{fig5.5}
\end{figure}

\begin{figure}[!tb]
\centering
\includegraphics[width = 0.6\textwidth]{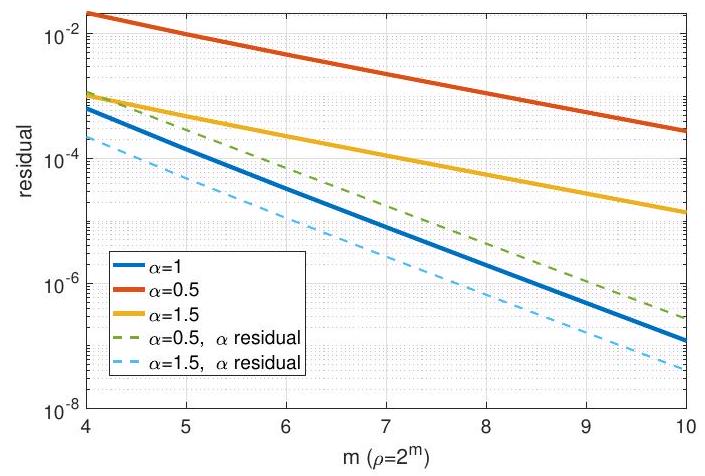}
\caption{
Illustration of the order of the residual $r_{k}$ and the $\alpha$-residual $r_{k}^{\alpha}$ defined in~\eqref{eq4.11}.}
\label{fig5.6}
\end{figure}

\paragraph{The scale of the residual $r_{k}$ at $\alpha=1$ and $\alpha \neq 1$.}
Finally, note that Proposition~\ref{prop4.1} in the proof of Theorem~\ref{theo3.2} and its continuous counterpart~\eqref{eq3.15} in Section~\ref{sec3.3} indicate that the residual $r_{k}$ is $\mathcal{O}(\epsilon^{2})$ if $\alpha=1$ and $\mathcal{O}(\epsilon)$ if $\alpha \neq 1$.
To numerically investigate the scaling of the residual $r_{k}$, we study both its expectation and its standard deviation.
Figure~\ref{fig5.5} plots the two quantities $\epsilon^{-1} \mathbb{E}(r_{k})$ and $\epsilon^{-1} \operatorname{std}(r_{k})$ when $\alpha \neq 1$ against the time $t_{k}=k \epsilon$ to verify that both $\mathbb{E} r_{k}$ and $\operatorname{std} r_{k}$ are of order $\epsilon^{-1}$ when $\alpha \neq 1$.
When $\alpha=1$, we plot $\epsilon^{-2} \mathbb{E}(r_{k})$ and $\epsilon^{-2} \operatorname{std}(r_{k})$ to show that both $\mathbb{E} r_{k}$ and std $r_{k}$ are of order $\epsilon^{-2}$.
Moreover, as mentioned in Remark~\ref{rema4.1}, the new definition of the residual, $\alpha$-residual $r_{k}^{\alpha}$ is $\mathcal{O}(\epsilon^{2})$ regardless of $\alpha$, as confirmed numerically by Figure~\ref{fig5.6}.

\subsection{Generalized ridge and lasso regression}\label{sec5.2}
In this subsection, we perform experiments on the generalized ridge/lasso regression:
\begin{equation}\label{eq5.4}
\begin{aligned}
&
\operatorname{minimize}_{x \in \mathbb{R}^{d}, z \in \mathbb{R}^{m}}
&&
\frac{1}{2} \mathbb{E}_{\boldsymbol{\xi}}(\xi_{in}^{\top} x-\xi_{obs})^{2}+g(z)
\\
&
\text{subject to}
&&
A x-z=0
\end{aligned}
\end{equation}
where $g(z)=\frac{1}{2} \beta\|z\|_{2}^{2}$ (ridge regression) or $g(z)=\beta\|z\|_{1}$ (lasso regression), with a constant $\beta>0$.
$A$ is a penalty matrix specifying the desired structured pattern of $x$.
Among the random variables $\xi=\left(\xi_{i n}, \xi_{obs}\right) \in \mathbb{R}^{d+1}, \xi_{i n}$ is a zero-mean random (column) vector uniformly distributed over the hyper-cube $[-0.5,0.5]^{d}$ with independent components.
The labelled output data $\xi_{obs}:=\xi_{i n}^{\top} v+\zeta$, where $v \in \mathbb{R}^{d}$ is a known vector and $\zeta=\mathcal{N}(0, \sigma_{\zeta}^{2})$ is the zero-mean measurement noise which is independent of $\xi_{i n}$.
For this regression problem, we have
$$
f(x, \boldsymbol{\xi})
=
\frac{1}{2}|\xi_{i n}^{\top} x-\xi_{obs}|^{2}
=
\frac{1}{2}|\xi_{i n}^{\top}(x-v)-\zeta|^{2}
$$
and
$$
f(x)
=
\mathbb{E} f(x, \boldsymbol{\xi})
=
\frac{1}{2}(x-v)^{\top} \Omega(x-v)+\frac{1}{2} \sigma_{\zeta}^{2}
$$
where $\Omega$ is the covariance matrix $\Omega=\mathbb{E}\left[\xi_{i n} \xi_{i n}^{\top}\right]=\frac{1}{12} I$.
Thus, $f^{\prime}(x, \boldsymbol{\xi})=\xi_{i n} \xi_{i n}^{\top}(x-v)-\xi_{i n} \zeta$ and $f^{\prime}(x)=\mathbb{E} f^{\prime}(x, \boldsymbol{\xi})=\Omega(x-v)$ The true solution of the ridge regression with $g(z)=$ $\frac{1}{2} \beta\|z\|_{2}^{2}$ is $x_{*}=\left(\Omega+\beta A^{\top} A\right)^{-1} \Omega v$.
The analytic expression of the diffusion matrix $\Sigma(x)$ can be written in terms of the fourth-order momentum of the stochastic vector $\xi_{i n}$:
$$
\Sigma(x)
=
\mathbb{E}\left[((x-v)^{\top} \xi_{in})^{2} \xi_{in} \xi_{in}^{\top}\right]
-
\Omega(x-v)(x-v)^{\top} \Omega+\sigma_{\zeta}^{2} \Omega
$$

As mentioned at the beginning of this Section~\ref{sec5}, we can use a batch size $B$ for the stochastic ADMM.
Then the corresponding SME~\eqref{eq5.2} for the ridge regression problem is
$$
\widehat{M} d X_{t}
=
-\Omega\left(X_{t}-v\right) d t
-\beta A^{\top} A X_{t} d t
+\sqrt{\frac{\epsilon}{B}} \Sigma^{1 / 2}(X_{t}) d W_{t}
$$
The stochastic modified differential inclusion for the lasso regression (formally) is
$$
\widehat{M} d X_{t}
\in
-\Omega\left(X_{t}-v\right) d t
-\frac{1}{2} \beta A^{\top} \operatorname{sign}\left(A X_{t}\right) d t
+\sqrt{\frac{\epsilon}{B}} \Sigma^{1 / 2}(X_{t}) d W_{t}
$$
To solve these SDEs, the main computational bottleneck is the computation of the matrix $\Sigma(x)$.
We can reduce this overhead by using a sample version for the covariance matrix $\Sigma(x)$ as shown~\eqref{eq5.1}.
At each $x, \Sigma_{N}(x)$ is computed by $N$ samples of the random $\xi$.
A small $N$ has a large gain in computational efficiency with the potential negative impact on the accuracy.
However, our extensive numerical tests show that even for a small $N$, the weak approximation of $X_{t}$ is still extraordinarily good.
$N=9$ is used for our results reported here when the numerical solution of the SME is involved.

\begin{figure}[!tb]
\centering
\begin{subfigure}[b]{0.45\textwidth}
\centering
\includegraphics[width = \textwidth]{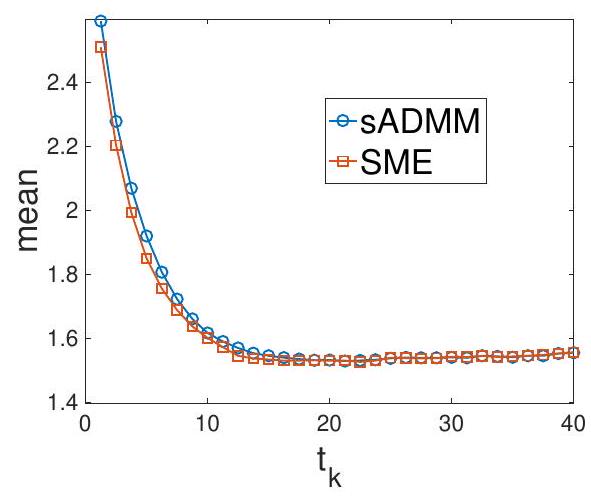}
\caption{$m=5$ (i.e., $\rho=2^{m} / T=0.8, \epsilon=\rho^{-1}=1.25$)}
\end{subfigure}
\hfill
\begin{subfigure}[b]{0.45\textwidth}
\centering
\includegraphics[width = \textwidth]{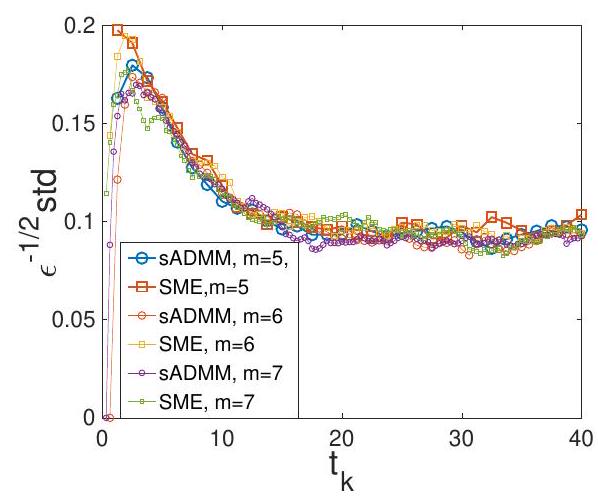}
\caption{The rescaled standard deviation of $\varphi(x_{k})$ and $\varphi(X_{k \epsilon})$ for $m=5,6,7$}
\end{subfigure}
\caption{
The mean and standard deviation of $\varphi(x_{k})$ (sADMM) and $\varphi(X_{k \epsilon})$ (SME). $\alpha=1.5$, $c=1$, $\omega=1$, $\omega_{1}=1$; $T=40$; $g(z)=\frac{1}{2} \beta\|z\|_{2}^{2}$. The results are based on 400 independent runs.}
\label{fig5.7}
\end{figure}

\begin{figure}[!tb]
\centering
\begin{subfigure}[b]{0.45\textwidth}
\centering
\includegraphics[width = \textwidth]{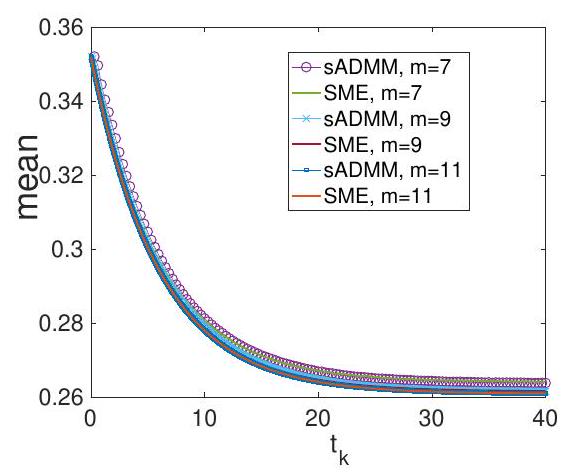}
\caption{The mean of $\varphi(x_{k})$ (sADMM) and $\varphi(X_{k \epsilon})$ (SME)}
\end{subfigure}
\hfill
\begin{subfigure}[b]{0.45\textwidth}
\centering
\includegraphics[width = \textwidth]{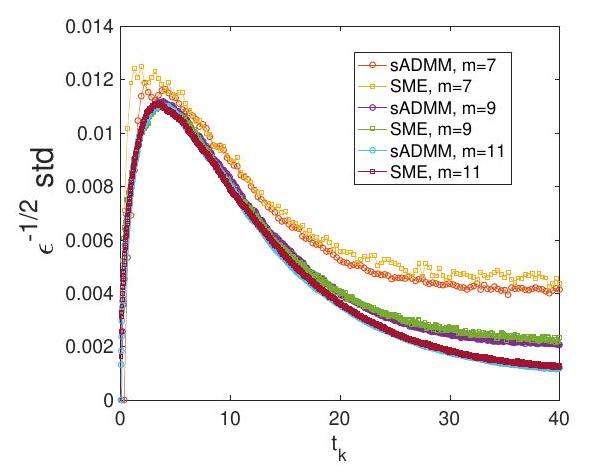}
\caption{The rescaled standard deviation of $\varphi(x_{k})$ and $\varphi(X_{k \epsilon})$}
\end{subfigure}
\caption{
The mean and standard deviation of $\varphi(x_{k})$ (sADMM) and $\varphi(X_{k \epsilon})$ (SME). $\alpha=1.5$, $c=1$, $\omega=1$, $\omega_{1}=1$; $T=40$; $g(z)=\frac{1}{2} \beta\|z\|_{1}$; $\epsilon=T 2^{-m}$. The results are based on 4000 independent runs.}
\label{fig5.8}
\end{figure}

In our experiments, we set the dimension $d=3$.
Let $A$ be the Hilbert matrix multiplied by 0.5.
Note that $A^{\top} A$ is quite close to singular since the minimal eigenvalue is only 1.8e$-6$.
Set $\sigma_{\zeta}^{2}=0.1, \beta=0.2$ and set the vector $v$ as linspace$(1,2, d)$.
Choose the initial $X_{0}=x_{0}$ as the zero vector, $z_{0}=A x_{0}$ and set $c=1$.
For the ridge regression, we choose the test function
$$
\varphi(x)=\sum_{i=1}^{d} \exp \left(-x_{(i)}\right)
$$
for ridge regression, which is the sum of all components.
For the lasso regression, we choose the objective function $f(x)+g(A x)$ as the test function:
$$
\varphi(x)=f(x)+\beta\|A x\|_{1}=\frac{1}{2}(x-v)^{\top} \Omega(x-v)+\frac{1}{2} \sigma_{\zeta}^{2}+\beta\|A x\|_{1}
$$
since the optimal solution is not unique.
The choice of the step size $\epsilon=2^{-m} T$ with $T=40$ and an integer $m$.

\begin{figure}[!tb]
\centering
\includegraphics[width = 0.6\textwidth]{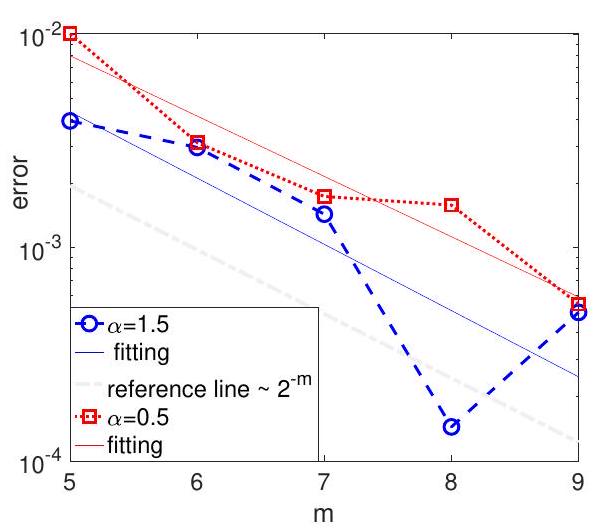}
\caption{
The weak convergence of order 1. The error is $\left|\mathbb{E}[\varphi(x_{\lfloor T \rho\rfloor})-\varphi(X_{T})]\right|$. Ridge regression. 4000 independent runs. Other parameters are the same as in Figure~\ref{fig5.7}; $\epsilon=T 2^{-m}$. The reference line has the exact slope $-1$ in the semi-log plot, corresponding to the exact relation that error $=\text{const}\times \epsilon$.}
\label{fig5.9}
\end{figure}

\begin{figure}[!tb]
\centering
\begin{subfigure}[b]{0.45\textwidth}
\centering
\includegraphics[width = \textwidth]{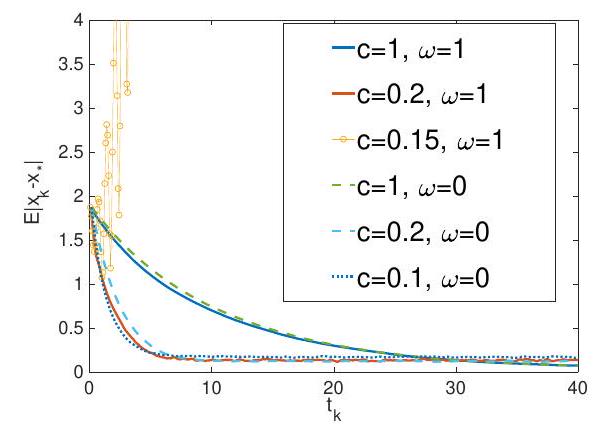}
\caption{The mean of the error.}
\end{subfigure}
\hfill
\begin{subfigure}[b]{0.45\textwidth}
\centering
\includegraphics[width = \textwidth]{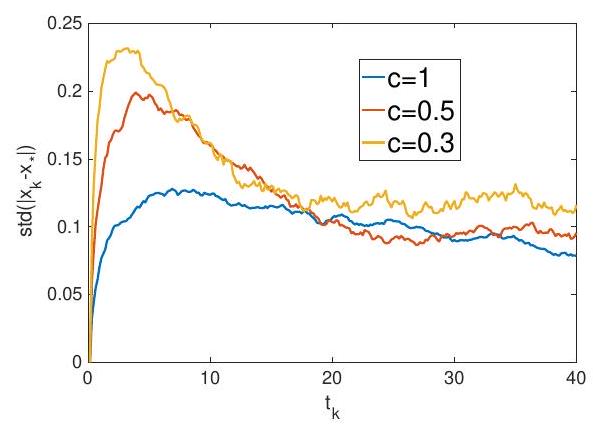}
\caption{The std of the error for $\omega=0$.}
\end{subfigure}
\caption{
With the fixed $\alpha=1.5$, the convergence/divergence of $x_{k}$ from stochastic ADMM for $\omega=1$ and $\omega=0$ with different choices of the parameter $c=\tau / \rho$. The average and standard deviation of the error $\|x_{k}-x_{*}\|$ are based on 400 independent runs. $\epsilon= T / 2^{8}$, $T=40$, i.e., $\rho=1 / \epsilon=6.4$. $g(z)=\frac{1}{2} \beta\|z\|_{2}^{2}$. At $c=0.15$ (less than the critical value $0.165$) and $\omega=1$, the matrix $\widehat{M}$ has one negative eigenvalue $-0.0153$ and the ADMM trajectory $x_{k}$ diverges.}
\label{fig5.10}
\end{figure}

\begin{figure}[!tb]
\centering
\begin{subfigure}[b]{0.45\textwidth}
\centering
\includegraphics[width = \textwidth]{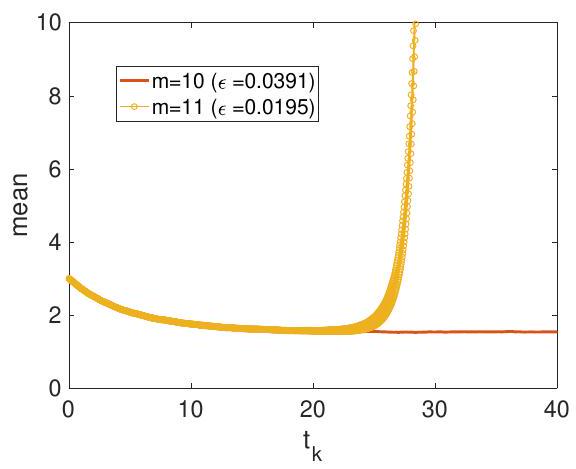}
\caption{The mean of $\varphi(x_{k})$.}
\end{subfigure}
\hfill
\begin{subfigure}[b]{0.45\textwidth}
\centering
\includegraphics[width = \textwidth]{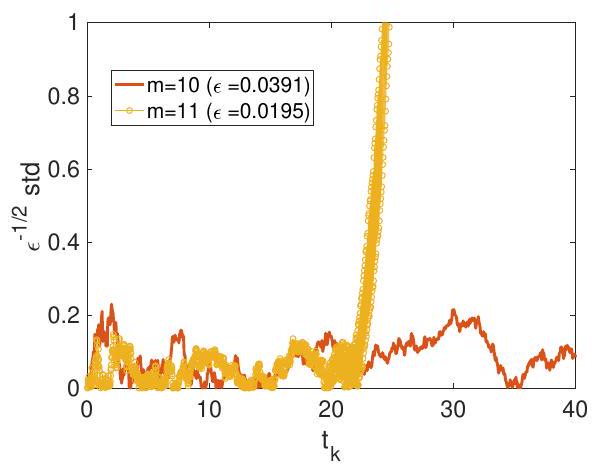}
\caption{The std of $\varphi(x_{k})$.}
\end{subfigure}
\caption{
$\alpha=2.02$. The trajectory $x_{k}$ from stochastic ADMM almost surely diverges if the stepsize $\epsilon=T / 2^{m}$ is small enough. $\omega_{1}=1$ and $\omega=0$. $c=\tau / \rho=1.400$ independent runs. $T=40$.}
\end{figure}
\begin{center}
\label{fig5.11}
\end{center}

\paragraph{Consistence between G-sADMM and SME.}
We show in Figure~\ref{fig5.7}a the mean of $\varphi(x_{k})$ and $\varphi(X_{k \epsilon})$ versus the time $t_{k}=k \epsilon$, for a moderately large $\epsilon=1.25$ at $m=5$ ($\rho$ is as small as 0.8$)$.
The results agree very well even for this small $\rho<1$.
To test the match of the fluctuation, we plot in Figure~\ref{fig5.7}b the rescaled standard deviation, $\epsilon^{-1 / 2}$ std, from both the stochastic ADMM and the continuous model for three different values of $\epsilon=2^{-m} T$ with $m=5,6,7$.
This figure confirms the scaling of std is $\mathcal{O}(\sqrt{\epsilon})$ and the weak convergence between the $x_{k}$ and $X_{k \epsilon}$.
It is worth pointing out that the values of $\rho$ are not large, only $0.8,1.6,3.2$, even though in the weak convergence, the limit of $\rho \rightarrow \infty$ is considered in Theorem~\ref{theo3.2}.

The results for lasso regression are also shown in Figure~\ref{fig5.8}, where we see the agreement for the means from stochastic ADMM and the SME and the agreement of the standard deviations for any fixed $\epsilon$.
This suggests that Theorem~\ref{theo3.2} may hold too even for non-smooth $g$.
However, the (rescaled) standard deviations for different $\epsilon$ do not overlap, so the scaling of $\operatorname{std}(x_{k})$ is $\mathcal{O}(\sqrt{\epsilon})$ may break down, in the case of a non-smooth $g$.

\paragraph{Weak convergence.}
The numerical verification of the weak convergence is presented in Figure~\ref{fig5.9} for the ridge regression, where $g$ is smooth and satisfies the regularity assumption in our theory.

\paragraph{The effect of parameter $c, \alpha$.}
Recall $c=\tau / \rho$, so $c$ determines the choice of the parameter $\tau$ used in G-sADMM~\eqref{eq1.6a}---\eqref{eq1.6c}.
Note $\widehat{M}=c I+(1 / \alpha-\omega) A^{\top} A$ plays the important role in the SME~\eqref{eq3.9}.

\paragraph{Over-relaxation $\alpha \in(1,2)$.}
Figure~\ref{fig5.10}a studies the behavior of convergence towards the true minimum point $x_{*}$ with the over-relaxation $\alpha=1.5$.
Then for $\omega=1$, with a decreasing value of $c, \widehat{M}$ stops being positive definite when $c<0.165$; while for $\omega=0$, $\widehat{M}$ is always positive definite for any $c>0$.
Figure~\ref{fig5.10} a shows the consistency of the convergence with the positive definiteness of $\widehat{M}$.%
\footnote{It is worthwhile noting that one needs a sufficient small $\epsilon$ to see this consistency between the discrete and the continuous-time model: for example if $m=4(\epsilon=1.25$ and $\rho=0.4), c=0.2(>0.165)$ leads to the divergence instead of convergence.
}
Figure~\ref{fig5.10}b shows the standard deviation when $c$ varies.
Our observation here for the over-relaxation is that a smaller $c$ can help decay faster (acceleration) on the expectation, but the fluctuation in stochastic ADMM also rises.
Figure~\ref{fig5.10}b at large time $(t>20)$ suggests that when the solution approaches the equilibrium (i.e., in saturation), the fluctuation is then independent of the choice of $c$.

\paragraph{Under-relaxation $\alpha \in(0,1)$.}
For the under-relaxation, $1 / \alpha-\omega>0$ for any $\omega \in[0,1]$ and thus $\widehat{M}$ is always positive definite, whose smallest eigenvalue is very close to $c$.
However, we observed in tests that $x_{k}$ diverges regardless of how small $\epsilon$ is if $c$ is less than 0.16.
For sufficient large $c(>0.16), x_{k}$ converges.
This suggests there are some myths about the convergence in this under-relaxation case: $\widehat{M}$ should be stronger than positive definite with some small but strictly positive lower bound for the eigenvalues to ensure the convergence.
Our current Corollary~\ref{coro3.1} and the continuous model in Theorem~\ref{theo3.2} are insufficient to answer this myth, which will be a future work.

\paragraph{Other choice of $\alpha \notin(0,2)$.}
In addition, for any negative $\alpha$ or $\alpha>2$, the stochastic ADMM scheme always diverges if $\epsilon$ is sufficiently small (i.e.~$\rho$ is sufficiently large), which is in alignment with Corollary~\ref{coro3.1} about the asymptotics of residual in the small $\epsilon$ regime: the residual converges to zero if and only if $\alpha \in(0,2)$.
It is important to note that if $\epsilon$ is not small enough, the ADMM sequence may also converge for $\alpha \notin(0,2)$.
As an illustration, we set $\alpha=2.02$ and compute the stochastic ADMM with two different stepsize $\epsilon$.
Figure~\ref{fig5.11} shows that for the smaller stepsize, the stochastic ADMM actually diverges since the corresponding continuous model is not stable.

\subsection{Discussion on adaptive parameter adjustment for step size, batch size, $c$ and $\alpha$}\label{sec5.3}
We showed above the SME formulation for understanding the dynamics of the stochastic ADMM both in theory and in numerics.
We now turn to a more interesting and challenging issue of how Theorem~\ref{theo3.2} can help design better adaptive schemes in practice.
In general, the main idea is based on some optimal control problem for the derived continuous-time model~\cite{LTE2017SME}.
The seek of the solvable and simple results of underlying control problems is the main endeavor to make the algorithm work efficiently in practice.
We here mainly focus on the one-dimensional case $d=1$ to present the idea and the extension to high dimensional problems might be carried out by imposing a local diagonal-quadratic assumption.

\paragraph{Two-phase behavior.}
The two-phase behavior of training convex functions with stochastic gradient is well known~\cite{NIPS2011_4316,LTE2017SME}.
The transition time $t_{*}$ is defined by the time when $\mathbb{E}\left|X_{t_{*}}-x_{*}\right|=\sqrt{\operatorname{Var} X_{t_{*}}}$.
Before $t_{*}$, the descent dominates and after $t_{*}$, the fluctuation dominates.
For a solvable SDE, as an illustration, we consider $d X_{t}=-X_{t} d t+\sqrt{\epsilon} \sigma d W_{t}$, which gives $t_{*}=\frac{1}{2} \log \left(2 x_{0}^{2} / \sigma^{2} \epsilon+1\right)$.

\paragraph{Adaptive batch size.}
The stochastic scheme~\eqref{eq1.6a}---\eqref{eq1.6c} uses only one instance of the gradient $f^{\prime}(x, \xi_{k+1})$ in each iteration.
If $f^{\prime}(x, \xi_{k+1})$ is replaced by the average $\frac{1}{B_{k+1}} \sum_{i=1}^{B_{k+1}} f^{\prime}(x, \xi_{k+1}^{i})$ where $B_{k+1}>1$ is the batch size and $(\xi_{k+1}^{i})$ are $B_{k+1}$ i.i.d.~samples, then $\Sigma$ should be multiplied by a factor $\frac{1}{B_{t}}$ where the continuous-time function $B_{t}$ is the linear interpolation of $B_{k}$ at times $t_{k}=k \epsilon$.
The stochastic modified equation~\eqref{eq3.9} is then in the following form $\widehat{M} d X_{t}=-\nabla V(X_{t}) d t+\sqrt{\frac{\epsilon}{B_{t}}} \sigma(X_{t}) d W_{t}$.
We consider the solvable SDE above again, but with adaptive $B_{t}: d X_{t}=-X_{t} d t+\sqrt{\epsilon / B_{t}} \sigma d W_{t}$.
Then before $t_{*}$, we can use a constant $B_{t}=B_{0}$.
After time $t>t_{*}$, our goal is to control the fluctuation $\sqrt{\operatorname{Var} X_{t_{*}}} \mathbb{E}$ by increasing $B_{t}$ to be as small as $\mathbb{E}\left|X_{t_{*}}-x_{*}\right|$.
For the solvable $\mathrm{SDE}$, this leads to the choice of $B_{t}$ as $B_{t}=B_{0} \frac{\sigma^{2} \epsilon}{2 X_{t_{*}}}\left(e^{2\left(t-t^{*}\right)}-1\right)$, for $t>t_{*}$.

\paragraph{Adaptive step size and adaptive $\widehat{M}$.}
The adaptive step size $\epsilon_{k}=1 / \rho_{k}$ can be linear interpolated to a continuous-time step size $\epsilon_{t}$ written in the form of $\epsilon_{0} u_{t}$ where $\epsilon_{0}$ is the maximal possible step size with a fixed $\epsilon_{0}$ and $u_{t} \in(0,1)$.
Then~\eqref{eq3.9} becomes $
\widehat{M} d X_{t}
=
-u_{t} \nabla V(X_{t}) d t+u_{t} \sqrt{\frac{\epsilon_{0}}{B_{t}}} \sigma(X_{t}) d W_{t}
$.
This step size adaptivity is exactly equivalent to the adaptive choice of $\widehat{M}_{t}=\widehat{M}_{0} / u_{t}$ with a fixed $\epsilon=\epsilon_{0}$:
\begin{equation}\label{eq5.5}
\frac{1}{u_{t}} \widehat{M}_{0} d X_{t}=-\nabla V(X_{t}) d t+\sqrt{\frac{\epsilon}{B_{t}}} \sigma(X_{t}) d W_{t}
\end{equation}
The adaptive choice of $\widehat{M}_{t}$ can be implemented by the adaptive choice for the ADMM parameters $c_{t}=\tau_{t} / \rho$ and relaxation parameter $\alpha_{t}$.

Our goal is to find an adaptive control rule of $u_{t}$ to solve $\min _{u:[0, T] \rightarrow[0,1]} \mathbb{E} V\left(X_{T}\right)$ subject to~\eqref{eq5.5}.
We consider a one dimensional quadratic objective function $V(x)=\frac{1}{2} a(x-b)^{2}$ and assume that $\sigma(x)=\sigma$ a positive constant.
The optimal control for this problem in the feedback form is $u_{t}^{*}=\min \left(1, \frac{2 \mathbb{E} V(X_{t})}{\epsilon \sigma^{2}}\right)$~\cite{LTE2017SME,Fleming2012Book}.
This policy means that if the fluctuation $\epsilon \sigma^{2}$ is small compared to the objective value $2 \mathbb{E} V(X_{t})$ (in the early stage of training), then $u_{t}^{*}$ reaches its ceiling 1 and we use the minimal $\widehat{M}_{t}=\widehat{M}_{0}$.
To construct the minimal $\widehat{M}_{0}$, one can set $\alpha=2$ for example (or additionally $c=0$ if $\omega=0$).
But when the fluctuation starts to dominate near the optimum, we should use a larger $\widehat{M}_{t}$.
Increasing $\widehat{M}_{t}$ can be done by increasing $c$ or decreasing $\alpha$ (or $\omega=0$ preferred).
The feedback control $u_{t}^{*}$, after being plugged back in~\eqref{eq3.9}, leads to the open-loop result $u_{t}^{*}=\min \left(1, \frac{2 \mathbb{E} V(X_{t})}{\epsilon \sigma^{2}}\right)=\frac{1}{1+a\left(t-t^{*}\right)}$ and consequently $\widehat{M}_{t}=\widehat{M}_{0}\left(1+a\left(t-t^{*}\right)\right)$ when $t>t_{*}$.
Note in a special case $\alpha=\omega=1$, then $c_{t}=\widehat{M}_{t}$ and the schedule for $\tau$ is $\tau_{t}=\tau_{0}\left(1+a(t-t^{*})\right)$.

\section{Conclusion and discussion on future works}\label{sec6}
In this paper, we have developed the stochastic continuous dynamics in the form of a stochastic modified equation (SME) to analyze the dynamics of a general family of stochastic ADMM, including the standard, linearized and gradient-based ADMM with relaxation $\alpha$.
Our continuous model~\eqref{eq3.9} provides a unified framework to describe the dynamics of stochastic ADMM algorithms and particularly can quantify the fluctuation effect for a large penalty parameter $\rho$.
Our analysis here generalizes the existing works~\cite{francca2018admm,pmlr-v97-yuan19c} limited in ordinary differential equation, and reveals the impact of the noise in the stochastic ADMM.
As the first-order approximation to the stochastic ADMM trajectory, the solution to the continuous model can precisely describe the mean and the standard deviation (fluctuation) of stochastic trajectories generated by stochastic ADMM.

One distinctive feature between the ADMM and its stochastic or online variants is highly similar to that between gradient descent and stochastic gradient descent~\cite{NIPS2011_4316,LTE2017SME,shamir2013stochastic}: there exists a transition time $t_{*}$ after which the fluctuation (with a typical scale $\rho^{-1 / 2}$) starts to dominate the "drift" term, which means the traditional acceleration methods in deterministic case will fail to perform well, despite of their prominent performance in the early stage of training when the drift suppresses the noise.
Our Section~\ref{sec5.3} extensively discusses this issue and the various aspects of the potential practicality of our theory: we briefly discuss a few new adaptive strategies such as for parameters $\rho_{t}, \tau_{t}$, the batch size $B_{t}$ and even relaxation parameters $\alpha_{t}$.
For example, a large $\alpha$ (over-relaxation) is preferred before the transition time while a small $\alpha$ (under-relaxation) may be preferred later to help reduce the variance.

For the possible further development in theory, we first note that we require the differential function $f$ and $g$ in our current assumption.
With the tool of stochastic differential inclusion and set-value stochastic process~\cite{SDI2003book}, we may also justify the given stochastic differential equation to allow the non-smooth $g$.
In addition, our family of ADMM methods is still restricted to the linearized ADMM and the gradient-based ADMM.
Recently, there emerge many new efficient acceleration methods for stochastic ADMM~\cite{NIPS2019_8955,pmlr-v32-zhong14,pmlr-v97-huang19a,pmlr-v28-ouyang13} and the stochastic variation reduction gradient in the context of ADMM~\cite{ADMM-Kwok-2016,VRSADMM-Cheng-2017,YuHuang2017}.
A potential future work might be to extend our stochastic analysis to these broad families for a better understanding of their properties at the continuous level.

\bibliography{references}
\bibliographystyle{alpha}

\end{document}